\documentclass[a4paper,10pt]{article}
\usepackage[utf8]{inputenc}
\usepackage{epsfig}
\usepackage[centertags]{amsmath}
\usepackage{amssymb, textcomp}
\usepackage{amsthm}
\usepackage{url}
\usepackage{comment}

\newtheorem{lemma}{Lemma}[section]
\newtheorem{corollary}[lemma]{Corollary}

\newtheorem{remark}[lemma]{Remark}

\newtheorem{proposition}[lemma]{Proposition}
\newtheorem{theorem}[lemma]{Theorem}

\def\dps{\displaystyle}

\def\al{\alpha}

\def\eps{\varepsilon}

\def\R{\mathbb R}

\def\CC{\mathbb C}

\def\cX{\mathcal X}
\def\EE{\mathcal E}
\def\CCC{\mathcal C}

\def\DDD{\mathcal D}

\def\MM{\mathcal M}

\def\FF{\mathcal F}

\def\GG{\mathcal G}

\def\XX{\mathcal X}
\def\RRR{\mathcal R}
\def\AAA{\mathcal A}

\def\TTT{\mathcal T}

\def\OO{\mathcal O}
\def\KK{\mathcal K}
\def \SSS{\mathcal S}

\def\B{\mathcal B}
\def\tB{\mathtt B}
\def\tA{\mathtt A}
\def\tC{\mathtt C}
\def\tD{\mathtt D}

\def\~{\tilde}

\def\de{\delta}
\def\ii{i}
\def\rr{\rho}

\def\wt{\widetilde}

\def\wh{\widehat}
\def\kk{\kappa}
\def\eps{\varepsilon}
\def\pa{\partial}

\def\d{\delta}
\def\max{\mathrm{max}}

\newcommand{\xb}{\bar{x}}
\newcommand{\yb}{\bar{y}}
\newcommand{\zb}{\bar{z}}
\newcommand{\tb}{\bar{t}}

\newcommand{\het}{z_0}
\newcommand{\uns}{\mathrm{u}}
\newcommand{\sta}{\mathrm{s}}

\newcommand{\h}{h_0}

\title{Breakdown of heteroclinic connections in the analytic Hopf-Zero singularity: Rigorous computation of the Stokes constant}

\author{Inmaculada Baldom\'a$^{1, 2, 3}$, Maciej J. Capi\'nski$^{4}$, Marcel Guardia$^{1,2, 3}$, Tere M. Seara$^{1,2, 3}$  \\
\small${}^{1}$ Departament de Matem\`atiques, Universitat Polit\`ecnica de Catalunya (UPC), Barcelona, Spain.\\
\small${}^{2}$ IMTECH, Universitat Polit\`ecnica de Catalunya (UPC), Barcelona, Spain.\\
\small${}^{3}$ Centre de Recerca Matem\`atica, Barcelona, Spain.\\
\small${}^{4}$AGH University of Science and Technology,
Faculty of Applied Mathematics\\
\small al. Mickiewicza 30, 30-059 Kraków, Poland}

\begin{document}

\maketitle

\begin{abstract}
    Consider analytic generic unfoldings of the three dimensional conservative Hopf-Zero singularity. Under open conditions on the parameters determining the singularity, the unfolding possesses two saddle-foci when the unfolding parameter is small enough. One of them has one dimensional stable manifold and two dimensional unstable manifold whereas the other one has  one dimensional unstable manifold and two dimensional stable manifold. Baldom\'a, Castej\'on and Seara \cite{BaCaSe13} gave an asymptotic formula for the distance between the one dimensional invariant manifolds in a suitable transverse section. This distance is exponentially small with respect to the perturbative parameter, and it depends on what is usually called a  Stokes constant. The non-vanishing of this constant implies that the distance between the invariant manifolds at the section is not zero. However, up to now there do not exist analytic techniques to check that condition. 
    %In this paper, we provide a computer assisted proof to give accurate estimates of the Stokes constant. We apply the method to two concrete unfoldings of the Hopf-Zero singularity.
    In this paper we provide a method for obtaining accurate rigorous computer assisted bounds for the Stokes constant. We apply it to two concrete unfoldings of the Hopf-Zero singularity, obtaining a computer assisted proof that the constant is non-zero
\end{abstract}

\tableofcontents

\section{Introduction}\label{sec:intro}
One of the fundamental questions in dynamical systems is to assess whether a given model possesses chaotic dynamics or not. In particular, one would like to prove whether the model has a hyperbolic invariant set whose dynamics is conjugated to the symbolic dynamics  of the usual Bernouilli shift by means of the construction of a Smale horseshoe. Since the pioneering works by Smale and Shilnikov, it is well known that the construction of such invariant sets may be attained by analyzing the stable and unstable invariant manifolds of hyperbolic invariant objects (critical points, periodic orbits, invariant tori) and their intersections.

Such analysis can be done by classical perturbative techniques such as (suitable versions of) Melnikov Theory~\cite{Melnikov63} or by means of Computer Assisted Proofs~\cite{MR3567489, MR3843291}. However, there are settings where Melnikov Theory nor ``direct'' Computer Assisted Proofs (that is, rigorous computation of the invariant manifolds) cannot be applied. For instance, in the so-called exponentially small splitting of separatrices setting. That is, on models which depend on a small parameter and where the distance between the stable and unstable invariant manifolds is exponentially small with respect to this parameter. 

This phenomenon of exponentially small splitting of separatrices often appears in analytic systems with different time scales, which couple fast rotation with slow hyperbolic motion. Example of such settings are nearly integrable Hamiltonian systems at resonances, near the identity area preserving maps or local bifurcations in Hamiltonian, reversible or volume preserving settings. In such settings, one needs more sophisticated techniques rather than Melnikov Theory to analyze the distance between the  stable and unstable invariant manifolds. Most of the results in the area follow the seminal approach proposed by Lazutkin in \cite{Lazutkin84} (there are though other approaches such as \cite{Treshev97}). Using these techniques, one can provide an asymptotic formula for the distance between the invariant manifolds, with respect to the perturbation parameter. If we denote by $\eps$ the small parameter, the distance is usually of the form 
\[
d=d(\eps)\sim \Theta \eps^\alpha e^{\frac{a}{\eps^\beta}}\qquad \text{as}\qquad \eps\to 0
\]
for some constants $\Theta$, $\alpha$, $a$ and $\beta$. In most of the settings, the constants $\alpha$, $a$ and $\beta$ have explicit formulas and can be ``easily'' computed for given models. However, the constant $\Theta$ is of radically different nature and much harder to compute. Indeed, the constants $\alpha$, $a$ and $\beta$ depend on certain first order terms of the model whereas $\Theta$, which we refer to as the Stokes constant, depends in a non-trivial way on the ``whole jet'' of the considered model. Note that it is crucial to know whether $\Theta$ vanishes or not, since its vanishing makes the whole first order between the invariant manifolds vanish and, consequently, chaos can not be guaranteed in the system. 

The purpose of this paper is to provide (computer assisted) methods to check, in given models, that  the Stokes constant does not vanish. Moreover, our method provides a rigorous accurate computation of this constant. To show the main ideas of the method and avoid technicalities, we  focus on the simplest setting where this method can be implemented: the breakdown of a one-dimensional heteroclinic connection for generic analytic unfoldings of the volume preserving Hopf-zero singularity. 

This problem was analyzed in \cite{BaSe08, BaCaSe13}. In these papers and the companions~\cite{MR3846870,MR3817789,BaIbSe20}, the authors prove that, in generic unfoldings of an open set of Hopf-zero singularities, one can encounter Shilnikov chaos \cite{Shil70}. The fundamental difficulty in these models is to prove that the one dimensional and two-dimensional heteroclinic manifolds connecting two saddle-foci in a suitable truncated normal form of the unfolding,  break down when one considers the whole vector field. These breakdowns, which are exponentially small,  plus some additional generic conditions lead to existence of chaotic motions.

\begin{remark}
A bifurcation with very similar behavior to that of the conservative Hopf-zero singularity is the Hamiltonian Hopf-zero singularity where a critical point of a 2 degree of freedon Hamiltoian system has a pair of elliptic eigenvalues and a pair of $0$ eigenvalues forming a Jordan block (see for instance \cite{GelfreichL14}). In generic unfoldings, the $0$ eigenvalues become a pair of small real eigenvalues and therefore the critical point becomes a saddle-center. In this setting, one can analyze the one dimensional invariant manifolds of the critical point and obtain an asymptotic formula for their distance (in a suitable section). This distance is exponentially with respect to the perturbative parameter. Then, to prove that they indeed do not intersect, one has to show that a certain Stokes constant is not zero as in the Hopf-zero conservative singularity. The methods presented in this paper can be adapted to this other setting. The Hamiltonian Hopf-zero singularity appears in many physical models, for instance in the Restricted Planar 3 Body Problem (see \cite{BaldomaGG21a,BaldomaGG21b}). It also plays an important role in the breakdown of small amplitude breathers for the Klein-Gordon equation (albeit in an infinite dimensional setting), see \cite{KS87,GomideGSZ21}. We plan to provide a computer assisted proof of the Stokes constant to guarantee the non-existence of small breathers in given Klein-Gordon equations in a future work.
\end{remark}

In this paper, we provide a method to compute the Stokes constant associated to the breakdown of the one-dimensional heteroclinic connection in analytic unfoldings of the conservative Hopf-zero singularity. 

Let us first explain the Hopf-zero singularity and state the main results about the breakdown of its one-dimensional heteroclinic connection obtained in  \cite{BaSe08, BaCaSe13}.

\subsection{Hopf-zero singularity and its unfoldings}\label{sec:introhopfZero}

The Hopf-zero  singularity takes place on a vector field 
$X^*:\mathbb{R}^3\rightarrow\mathbb{R}^3$, which has the origin as a critical point, and such that the eigenvalues of the linear part at this point are $0$, $\pm i\alpha^*$, for some $\alpha^*\neq0$.
Hence, after a linear change of variables, we can assume that the linear part of this vector field at
the origin is
$$
DX^*(0,0,0)=\left(\begin{array}{ccc}0 &\alpha^* &0\\-\alpha^* &0  &0\\0 &0 &0\end{array}\right).$$
We assume that $X^*$ is  analytic. Since $DX^*(0,0,0)$ has zero trace, it is reasonable to study it in the context of analytic conservative vector fields (see \cite{BV84} for the analysis of this singularity in the $\CCC^\infty$ class). In this case, the generic singularity can be met by a generic linear family depending on one parameter, and so it has codimension one.

We  study generic analytic families $X_{\mu}$ of conservative  vector fields on $\mathbb{R}^3$ depending on a parameter $\mu\in\mathbb{R}$, such that $X_{0}=X^*$, the vector field described above.

Following \cite{Guc81} and \cite{GH90}, after some changes of variables, we can write $X_{\mu}$ in its normal form up to order two, namely
\begin{eqnarray}\label{normalformcartabc}
 \frac{d\xb}{d\tb}&=&-\beta_1\xb\zb+\yb\left(\alpha^*+\alpha_2\mu+\alpha_3\zb\right)+\OO_3(\xb,\yb,\zb,\mu),\medskip\nonumber\\
\frac{d\yb}{d\tb}&=&-\xb\left(\alpha^*+\alpha_2\mu+\alpha_3\zb\right)-\beta_1\yb\zb+\OO_3(\xb,\yb,\zb,\mu),\medskip\\
\frac{d\zb}{dt}&=&-\gamma_0\mu+\beta_1\zb^2+\gamma_2(\xb^2+\yb^2)+\gamma_3\mu^2+\OO_3(\xb,\yb,\zb,\mu).\nonumber
\end{eqnarray}
Note that the coefficients $\beta_1$, $\gamma_2$ and $\alpha_3$ depend exclusively on the vector field $X^*$.

From now on, we will assume that $X^*$ and its unfolding $X_\mu$ satisfy the following generic conditions:
%\marginpar{MC: I glued $\beta_1$ and $\gamma_0$ conditions. Hope this is ok.}
\begin{equation}\label{condicionsX0}
 \beta_1\neq0,
 %, \qquad \gamma_1\neq 0 
 \qquad
\gamma_0\neq 0.
\end{equation}
%Moreover, we will consider unfoldings satisfying the generic conditions:
%$$
%\beta_0\neq0, \qquad
%\gamma_0\neq 0.
%$$
% In this case, redefining the parameters $\mu$ and $\nu$, we can assume that:
% \begin{equation}\label{condicionsunfolding1}
%  \beta_0=\gamma_0=1.
% \end{equation}
Depending on the other coefficients $\alpha_i$ and $\gamma_i$, one obtains different qualitative behaviors for the orbits of the vector field $X_{\mu}$.
We  consider  $\mu$ satisfying
\begin{equation}
\beta_1 \gamma_0\mu>0
 %,\qquad |\beta_0\nu|<|\beta_1|\sqrt{|\mu|}
 \label{condicionsparams}.
\end{equation}
In fact, redefining the parameters $\mu$  and the variable
$\zb$, one can achieve
\begin{equation}\label{condicionscompactes}
\beta_1>0,\qquad \gamma_0=1,
\end{equation}
and consequently the open set defined by \eqref{condicionsparams} is now
\begin{equation}\label{condicionsparamsconcr}
 \mu>0.
 \end{equation}
Moreover, dividing the variables $\xb,\yb$ and $\zb$ by $\sqrt{\beta_1}$, and scaling time by $\sqrt{\beta_1}$, redefining the coefficients and denoting $\alpha_0=\alpha^*/\sqrt{\beta_1}$, we can assume that $\beta_1=1$, and therefore system \eqref{normalformcartabc} becomes
\begin{eqnarray}\label{sistema}
 \frac{d\xb}{d\tb}&=&-\xb\zb+\yb\left(\alpha_0+\alpha_2\mu+\alpha_3\zb\right)+\OO_3(\xb,\yb,\zb,\mu),\medskip\nonumber\\
\frac{d\yb}{d\tb}&=&-\xb\left(\alpha_0+\alpha_2\mu+\alpha_3\zb\right)-\yb\zb+\OO_3(\xb,\yb,\zb,\mu),\medskip\\
\frac{d\zb}{dt}&=&-\mu+\zb^2+\gamma_2(\xb^2+\yb^2)+\gamma_3\mu^2+
\OO_3(\xb,\yb,\zb,\mu).\nonumber
\end{eqnarray}
We denote by $X_{\mu}^2$, usually called the normal form of second order, the vector field obtained considering the terms of \eqref{sistema} up to order two. Therefore, one has
$$
X_{\mu}=X_{\mu}^2+F_{\mu}^2, \quad \mbox{where } \ F_{\mu}^2(\xb,\yb,\zb)=
\OO_{3}(\xb,\yb,\zb,\mu).
$$
It can be easily seen that system \eqref{sistema} has two critical points at distance $\OO(\sqrt{\mu})$ to the origin.
Therefore, we scale the variables and parameters so that the critical points are $\OO(1)$ and not $\OO(\sqrt{\mu})$. That is, we  define the new parameter
$\delta=\sqrt{\mu}$, and the new variables $x=\delta^{-1}\xb$, $y=\delta^{-1}\yb$, $z=\delta^{-1}\zb$ and $t=\delta\tb$.
Then, renaming the coefficients $b=\gamma_2$, $c=\alpha_3$,
%and $\coef=\beta_1$,
 system \eqref{sistema} becomes
\begin{equation}\label{initsys}
 \begin{array}{rcl}
 \displaystyle\frac{dx}{dt}&=&\displaystyle - x  z+\left(\frac{\alpha(\delta^2)}{\delta}+cz\right)y+\delta^{-2}f(\delta x,\delta y, \delta z, \delta),\medskip\\
  \displaystyle\frac{dy}{dt}&=&\displaystyle-\left(\frac{\alpha(\delta^2)}{\delta}+cz\right)x-y z+\delta^{-2}g(\delta x,\delta y, \delta z, \delta),\medskip\\
  \displaystyle\frac{dz}{dt}&=&-1+b(x^2+y^2)+z^2+\delta^{-2}h(\delta x, \delta y, \delta z, \delta,),
 \end{array}
\end{equation}
where  $f$, $g$ and $h$ are real analytic functions of order three in all their variables, $\delta>0$ is a small parameter and  $\alpha(\delta^2)=\alpha_0+\alpha_2\de^2$.
%is an analytic function such that $\alpha(0)=\alpha_0\neq0$ and $\alpha'(0)=\alpha_2$.
%\textbf{I think that we can write $\alpha$ independent of $\d^2$ because the terms $O(\d)$ are included in $\delta^{-2} f$. In addition we can also skip $c$}

\begin{remark}
Without loss of generality, we can assume that $\alpha_0$ and $c$ are both positive constants. In particular, for $\delta$ small enough, $\alpha(\delta^2)$ will be also positive.
\end{remark}

Observe that, if we do not consider the higher order terms (that is, $f=g=h=0$), we obtain the unperturbed system
\begin{equation}\label{initsys0}
 \begin{array}{rcl}
  \displaystyle\frac{dx}{dt}&=&\displaystyle - x z+\left(\frac{\alpha(\delta^2)}{\delta}+cz\right)y,\medskip\\
  \displaystyle\frac{dy}{dt}&=&\displaystyle-\left(\frac{\alpha(\delta^2)}{\delta}+cz\right)x-y  z,\medskip\\
  \displaystyle\frac{dz}{dt}&=&-1+b(x^2+y^2)+z^2.
 \end{array}
\end{equation}
The next lemma gives the main properties of this system.

\begin{lemma}[\cite{BaCaSe13}]\label{lemma:unper}
For any value of $\delta>0$, the unperturbed system \eqref{initsys0} has the following properties:
\begin{enumerate}
 \item
 It possesses two hyperbolic fixed points $S_{\pm}^0=(0,0,\pm1)$ which are of saddle-focus type with eigenvalues $\mp 1+|\frac{\alpha}{\delta}\pm c|i$,
$\mp 1-|\frac{\alpha}{\delta}\pm c|i$, and $\pm2$.
\item
The one-dimensional unstable manifold of $S_{+}^0$ and the one-dimensional stable manifold of $S_{-}^0$ coincide along the heteroclinic connection
$\{(0,0,z):\, -1<z<1\}.$
The time parameterization of this heteroclinic connection is given by 
$$
\Upsilon_0(t)=(0,0,\het(t))=(0,0,-\tanh t),
$$
if we require $\Upsilon_0(0)=(0,0,0).$
\end{enumerate}
\end{lemma}
 Their $2$-dimensional stable/unstable manifolds also coincide, but we will not deal with this problem in this paper.

The critical points given in Lemma \ref{lemma:unper} are persistent for system \eqref{initsys} for small values of $\de>0$.  Below we summarize some properties of system \eqref{initsys}.
\begin{lemma}[\cite{BaCaSe13}]\label{lemaptscritics}
If $\delta>0$ is small enough, system \eqref{initsys} has two fixed points $S_{\pm}(\delta)$ of saddle-focus type,
$$
S_{\pm}(\delta)=(x_\pm(\delta), y_\pm(\delta),z_\pm(\delta)),
$$
with
$$
x_\pm(\delta)=\OO(\delta^2),\quad y_\pm(\delta)=\OO(\delta^2),\quad z_\pm(\delta)=\pm1+\OO(\delta).
$$
The point $S_+(\delta)$ has a one-dimensional unstable manifold and a two-dimensional stable one.
Conversely, $S_-(\delta)$ has a one-dimensional stable manifold and a two-dimensional unstable one.

Moreover, there are no other fixed points of \eqref{initsys} in the closed ball $B(\delta^{-1/3}).$
\end{lemma}

The theorem proven in \cite{BaCaSe13} is the following.
\begin{theorem}[\cite{BaCaSe13}]\label{maintheoremrescaled}
Consider system \eqref{initsys}, with $\delta>0$ small enough.
Then, there exists a constant $C^*$, such that the distance $d^{\uns,\sta}$ between the one-dimensional stable manifold of $S_-(\delta)$ and the one-dimensional unstable manifold of $S_+(\delta)$, when they meet the plane $z=0$, is given by
$$
d^{\uns,\sta}=\delta^{-2}e^{-\frac{\alpha_0\pi}{2\delta}}e^{\frac{\pi}{2}(\alpha_0h_0+c)}\left(C^*+\OO\left(\frac{1}{\log(1/\delta)}\right)\right),
$$
where $\alpha_0=\alpha(0)$,  and $h_0=-\lim_{z\to 0} z^{-3} h(0,0,z,0,0)$.
\end{theorem}

In~\cite{BaCaSe13} was proven that the constant $C^*$ comes from the so called \textit{inner equation} and that, generically, it does not vanish. However, for a given model is usually very hard to prove analytically whether the associated $C^*$ vanishes or not. In this paper we provide a rigorous (computer assisted) method to check whether it vanishes and to compute its value.

\subsection{The inner equation}
One of the key parts of the proof of Theorem \ref{maintheoremrescaled} is to analyze an inner equation. This equation provides the Stokes constant $C^*$ and it was obtained and analyzed in  \cite{BaSe08}. To obtain it, 
\begin{comment}
After normal form procedure and some scalings, a generic unfolding of the Hopf-zero singularity can be written as:
\begin{equation}\label{initsys}
 \begin{array}{rcl}
 \displaystyle\frac{dx}{dt}&=&\displaystyle - x  z+\left(\frac{\al}{\delta}+cz\right)y+\delta^{-2}f(\delta x,\delta y, \delta z, \delta),\medskip\\
  \displaystyle\frac{dy}{dt}&=&\displaystyle-\left(\frac{\al}{\delta}+cz\right)x-y z+\delta^{-2}g(\delta x,\delta y, \delta z, \delta),\medskip\\
  \displaystyle\frac{dz}{dt}&=&-1+b(x^2+y^2)+z^2+\delta^{-2}h(\delta x, \delta y, \delta z, \delta,),
 \end{array}
\end{equation}
where  $f$, $g$ and $h$ are real analytic functions of order three in all their variables, $\delta>0$ is a small parameter.
\end{comment}
we perform the change of coordinates
$(\phi,\varphi,\eta)=C_{\d}(x,y,z)$ given by
\begin{equation*}
\phi=\d(x+\ii y), \qquad \varphi = \d(x-\ii y),\qquad \eta = \d z, \qquad
\tau=\frac{t-\ii\pi/2}{\d}.
\end{equation*}
Applying this change to  system~\eqref{initsys}, one obtains
\begin{eqnarray}\label{almostinner}
\frac{d\phi}{d\tau}&=&\big (-\al\ii -\eta\big )\phi
+\tilde{F}_1(\phi,\varphi,\eta ,\d) \notag\\
\frac{d\varphi}{d\tau}&=&\big (\al \ii -\eta \big )\varphi
+\tilde{F}_2(\phi,\varphi,\eta,\d) \\
\frac{d\eta}{d\tau}&=& -\d^2 + b\phi\varphi +\eta^2 +
\tilde{H}(\phi,\varphi,\eta,\d) \notag
\end{eqnarray}
where
\begin{eqnarray*}
\tilde{F}_1(\phi,\varphi, \eta,\d)
&=& f(C_{\d}^{-1}(\phi,\varphi,\eta),\d)+\ii g(C_{\d}^{-1}(\phi,\varphi,\eta),\d),\\
\tilde{F}_2(\phi,\varphi, \eta,\d)
&=& f(C_{\d}^{-1}(\phi,\varphi,\eta),\d)-\ii g(C_{\d}^{-1}(\phi,\varphi,\eta),\d),\\
\tilde{H}(\phi,\varphi, \eta,\d) &=&h(C_{\d}^{-1}(\phi,\varphi,\eta),\d).
\end{eqnarray*}

The inner equation comes from~\eqref{almostinner} taking $\d=0$. Defining
$F_i(\phi,\varphi, \eta)=\tilde{F}_i(\phi,\varphi, \eta,0)$ and
$H(\phi,\varphi,\eta)=\tilde{H}(\phi,\varphi,\eta,0)$ and, for technical reasons, performing the change $\eta =-s^{-1}$, we get
\begin{eqnarray}\label{inner}
\frac{d\phi}{d\tau}&=&-\left (\al \ii - \frac{1}{s}\right )
+ F_1\left(\phi,\varphi,-s^{-1}\right) \notag\\
\frac{d\varphi}{d\tau}&=&\left (\al \ii + \frac{1}{s}\right )
+ F_2\left(\phi,\varphi,-s^{-1}\right) \label{eq:s-system}\\
\frac{ds}{d\tau}&=&1+s^2\left(b \phi \varphi +  H\left(\phi,\varphi,-s^{-1}\right)\right).\notag
\end{eqnarray}
We reparameterize time so that equation~\eqref{inner} becomes a
non-autonomous 2--dimensional equation with time $s$,
\begin{equation}\label{eq:Inner}
\begin{split}
\phi'&= \frac{-\left (\al \ii - \frac{1}{s}\right ) \phi
+ F_1(\phi,\varphi,-s^{-1})}{1+s^2(b \phi \varphi +  H(\phi,\varphi,-s^{-1}))}
\\
\varphi'&= \frac{\left ( \al \ii + \frac{1}{s}\right )\varphi
+ F_2(\phi,\varphi,-s^{-1})}{1+s^2(b \phi \varphi +  H(\phi,\varphi,-s^{-1}))}
\end{split}
\end{equation}
with $'=\frac{d}{ds}$.

To analyze this system, we separate its linear terms from the nonlinear ones. Indeed,  defining
\begin{equation}\label{def:MatrixA}
\AAA(s)= \begin{pmatrix}
-i\al+\frac{1}{s}&0\\
0&i\al+\frac{1}{s}
 \end{pmatrix}
\end{equation}
and
\begin{equation}\label{def:Rgen}
\SSS(\phi,\varphi,s)=
 \begin{pmatrix}
 \dps \frac{\left (\al \ii - \frac{1}{s}\right )\phi s^2\big (b \phi \varphi +  H(\phi,\varphi,-s^{-1})\big )+F_1(\phi,\varphi,-s^{-1}) }{1+s^2(b \phi \varphi +  H(\phi,\varphi,-s^{-1}))} \\
 \dps \frac{-\left (\al \ii + \frac{1}{s}\right )\varphi s^2 \big (b \phi \varphi +  H(\phi,\varphi,-s^{-1})\big )+F_2(\phi,\varphi,-s^{-1}) }{1+s^2(b \phi \varphi +  H(\phi,\varphi,-s^{-1}))}
 \end{pmatrix},
\end{equation}
equation \eqref{eq:Inner} can be expressed as
\begin{equation}\label{eq:InnerModifiedgen}
 \begin{pmatrix} \frac{d\phi}{ds}\\ \frac{d\varphi}{ds}\end{pmatrix}=\AAA(s) \begin{pmatrix}\phi\\ \varphi\end{pmatrix}+\SSS(\phi,\varphi,s).
\end{equation}
From now on we will refer to~\eqref{eq:InnerModifiedgen} as the inner equation.

%\subsection{Preliminary results}
For $s\in \CC$ we shall write $\Re s$ and $\Im s$ for its real and imaginary part, respectively. Following~\cite{BaSe08}, we define the  \emph{inner domains} as
\begin{equation}\label{def:domainInnner}
% \begin{split}
\DDD^-_\rr = \left\{ s \in \CC:
|\Im s|\geq -\tan \beta \Re s -\rr, \; \Re s\leq 0
\right\}, \qquad
\DDD^+_\rr=\{s:-s\in \DDD^-\}
%  = \DuInn \cap \DsInn \cap
% \left\{ U \in \complexs \st |\Im(U)|\geq 0
% \right\},
% \end{split}
\end{equation}
for some $\rr>0$.

%The following result is proven in~\cite{BaSe08}.
\begin{theorem}[\cite{BaSe08}]\label{th:inner}
If $\rr$ is big enough, the inner equation has two solutions
$\psi^\pm=(\phi^{\pm},\varphi^{\pm})$ defined in $\DDD^\pm_\rr$ satisfying the asymptotic condition
\begin{equation}\label{def:asymptcond}
 \lim_{\Re s\to\pm\infty}\psi^\pm(s)=0.
\end{equation}
Moreover its difference satisfies that, for $s\in \DDD^+_\rr\cap \DDD^-_\rr \cap \{\Im s<0\}$
\begin{equation}\label{def:asymptformulainner}
\Delta \psi(s) = \psi^{+}(s)-\psi^-(s) =  s e^{-i \al (s -  \h \log s
)}\left[ \left (\begin{array}{c}\Theta \\ 0\end{array}\right )+\mathcal{O}\left(\frac{1}{|s|}\right)\right].
\end{equation}
In addition $\Theta\neq 0$ if and only if $\Delta \psi\not\equiv 0$.
%\textcolor{red}{mention the domain of definition}
\end{theorem}
%\begin{remark}
%The result in~\cite{BaSe08} provides more information, but the above properties are enough for the moment.
%\end{remark}
Note that the Stokes constant $\Theta\in\mathbb{C}$ can be defined as
\begin{equation}\label{def:StokesLimit}
\Theta=\lim_{\Im s\to-\infty}s^{-1}e^{i\al( s - \h \log s)}\Delta\phi(s).
\end{equation}
Later, in \cite{BaCaSe13} the authors prove that, if $C^*$ is the constant introduced in Theorem~\ref{maintheoremrescaled}, then
$$
C^*=|\Theta|.
$$
However there is no closed formula for $\Theta$, which depends on the full jet of the nonlinear terms
in~\eqref{eq:InnerModifiedgen}. Our strategy to compute  $\Theta$  is to perform a computer assisted proof.

\section{Rigorous computation of the Stokes constant}\label{sec:method}
%We propose two methods to check that the  Stokes constant is not zero.
 %Both methods boil down to give an explicit $\rho_*$ such that the existence of $\psi^{\pm}$ in the domain $\{\Re z=0, \, \Im z \leq -\rr_*\}$ is guaranteed and compute $\Delta\psi(-i\rr^*)=\psi^+(-i\rr^*)-\psi^-(-i\rr^*)$.
  %\begin{itemize}
   % \item Method 1: By Theorem \ref{th:inner}, if one can check (relying on rigorous computer computations) that $\Delta\psi(-i\rr^*)=\psi^+(-i\rr^*)-\psi^-(-i\rr^*) \neq 0$, one can ensure that $\Theta\neq 0$.
    %\item Method 2: We provide a method to provide rigorous  accurate estimates for $\Theta$. This second method also relies on evaluating $\Delta\psi(-i\rr^*)$ but also takes more advantage on the fixed point argument techniques used to prove Theorem \ref{th:inner}. 
    %Notice that, if we can check that $\Delta\psi(-i\rr^*)=\psi^+(-i\rr^*)-\psi^-(-i\rr^*) \neq 0$ then the Stokes constant is not zero. Moreover, we will provide a method to compute $\Theta$.
%\end{itemize}

%In addition, we will explain also how to improve the accuracy of the computations.

%In  Section~\ref{sec:examples}, we apply these two methods to two examples and, finally,  in Section~\ref{sec:StokesComputeExtrap}, we explain how to improve the accuracy of Method 2 in one of the examples considered in Section~\ref{sec:examples}.

We propose a method to compute the Stokes constant $\Theta$ relying on rigorous, computer assisted, interval arithmetic based validation. The method 
takes advantage from the constructive method for proving Theorem~\ref{th:inner} based on fixed point arguments and we strongly believe that it can be applied to other settings as, for instance, the classical rapidly forced pendulum and close to the identity area preserving maps.

The method we propose to compute the Stokes constant $\Theta$ is divided into two parts.
\begin{itemize}
    \item Part 1:  
    We provide an algorithm to give an explicit $\rho_*>0$ such that the existence of the solutions $\psi^{\pm}$ of the inner equation, in the domain $\{\Re s=0, \, \Im s \leq -\rr_*\}$ is guaranteed. 
    
    The algorithm is based in giving explicit bounds (which depend on the nonlinear terms $\mathcal{S}$ of the inner equation, see \eqref{def:Rgen}) of all the constants involved in the fixed point argument. We believe that this algorithm can be generalized to other situations where the proof of the existence of the corresponding solutions of the inner equation relies on fixed point arguments. 
    
    In the case of the Hopf-zero singularity, by Theorem \ref{th:inner}, if one can check (using rigorous computer computations) that $\Delta\psi(-i\rr^*)=\psi^+(-i\rr^*)-\psi^-(-i\rr^*) \neq 0$ one can ensure that $\Theta\neq 0$.   
\item Part 2: Using that $\Delta \psi(s)$ is defined for $s\in \{\Re s=0,\, \Im s\leq -\rho_*\}$ with $\rho_*$ given in Part 1, we give a method which provides rigorous accurate estimates for $\Theta$.

    We give an algorithm to compute $\rho_{0}\geq \rho_*$ such that, for all $\rr\geq \rr_0$, the Stokes constant and $\Delta \phi(-i\rho)$ satisfies the relation
    \begin{equation}\label{formulaTheta0}
    \Theta = i \rr^{-1} \Delta \phi(-i\rho)e^{\alpha \left (\rho -i h_0 \log \rho - h_0\frac{\pi}{2}\right )} (1+g(\rho))
    \end{equation}
    with $|g(\rho)|<1$. By \eqref{def:asymptformulainner}, we know that  $|g(\rho)|$ is of order $\mathcal{O}(\rho^{-1})$. We provide explicit upper bounds for it.
    
    Part 2 also relies on evaluating $\Delta\psi(-i\rr)$ but takes more advantage on the fixed point argument techniques used to prove formula~\eqref{def:asymptformulainner} in Theorem~\ref{th:inner}. 
    
    A similar formula to~\eqref{formulaTheta0} for $\Theta$ can be deduced in other settings such as the rapidly forced pendulum and close to the identity area preserving maps.
    %, by adding an error term (depending on $\rho$) to the right hand side. Controlling this error term, 
    We should be able to adapt our method to a plethora of different situations. 
\end{itemize}

In Section~\ref{sec:theoretical} we show the theoretical framework we use to design the method. In particular, the functional setting needed for the fixed point argument.
It is divided in Sections~\ref{sec:theoreticalPart1}
and~\ref{sec:theoreticalPart2} which deal with Part 1 and Part 2 respectively. 
In Section~\ref{sec:method1}, we follow the theoretical approach given in the previous sections and compute all the necessary constants to implement the method. After that, in Section~\ref{sec:AlgorithmMethod1} we write the precise algorithm, pointing out all the constants that need to be computed to find $\Theta$. 
%}
%\textcolor{blue}{
In  Section~\ref{sec:examples}, we apply our method to two examples. Finally,  in Section~\ref{sec:StokesComputeExtrap}, we explain how to improve the accuracy in the computation of the Stokes constant in one of the examples considered in Section~\ref{sec:examples}.
%}
\subsection{Scheme of the method. Theoretical approach}\label{sec:theoretical}

\subsubsection{Existence domain of the solutions of the inner equation }\label{sec:theoreticalPart1}
We analyze the solutions $\psi^\pm=(\phi^\pm,\varphi^\pm)$  of equation \eqref{eq:InnerModifiedgen} in the inner domains $\DDD^\pm_\rr$ introduced in 
\eqref{def:domainInnner}.
% \begin{split}
%\DDD^-_\rr = \left\{ s \in \CC:
%|\Im s|\geq -\tan \beta \Re s -\rr, \; \Re s\leq 0
%\right\}, \qquad
%\DDD^+_\rr=\{s:-s\in \DDD^-\}
%%  = \DuInn \cap \DsInn \cap
%% \left\{ U \in \complexs \st |\Im(U)|\geq 0
% \right\},
% \end{split}
%\end{equation}
%for some $\rr>0$.
To prove the existence of the solutions $\psi^\pm$, we set up a fixed point argument. From now on we use subindices 1 and 2 to denote the two components of all vectors and operators.

Note that the right hand side of equation \eqref{eq:InnerModifiedgen} has a linear part plus higher order terms (which will be treated as perturbation). We consider a fundamental matrix $M(s)$
associated to the matrix $\AAA$ in \eqref{def:MatrixA} given by
\begin{equation}\label{def:MatrixM}
 M(s)=s \begin{pmatrix}
e^{-i\al s}&0\\
0&e^{i\al s}
 \end{pmatrix}
\end{equation}
and we define also the integral operators
\begin{equation}\label{def:OperatorB}
\B^{\pm}(h)=\begin{pmatrix}\B_1^{\pm}(h)\\ \B_2^{\pm}(h)\end{pmatrix}=M(s)\int_{\pm\infty}^0 M(s+t)^{-1}h(s+t)dt.
\end{equation}
Then, the solutions $\psi^\pm$ of equation \eqref{eq:InnerModifiedgen} satisfying the asymptotic conditions \eqref{def:asymptcond} must be also solutions of the integral equation
\[
 \psi^\pm=\B^\pm(\SSS(\psi,s)).
\]
Therefore, we  look for fixed points of the operators
\begin{equation}\label{def:F}
 \FF^\pm(\psi)=\B^\pm(\SSS(\psi,s)).
\end{equation}
%We perform all the analysis for $\psi^-$ in the domain $\DDD^-_\rr$. Note that the reversibility with respect to \eqref{def:involution} gives directly $\psi^+$  in the domain $\DDD^+_\rr$ with the same value $\rr$.
We define the Banach spaces
\begin{equation}\label{def:BanachChi}
 \XX_\nu^{\pm}=\left\{h: \DDD^{\pm}_\rr\to\CC: \text{analytic}, \|h\|_\nu<\infty\right\}\quad \text{where}\quad   \|h\|_\nu=\sup_{s\in \DDD^{\pm}_\rr}|s^\nu h(s)|.
\end{equation}
Then, we  obtain fixed points of the operators $\FF^\pm$ in the Banach spaces $\XX_\nu\times\XX_\nu$ with the norm
\[
 \|(\phi,\varphi)\|_{\nu}= \max\{\|\phi\|_\nu, \|\varphi\|_\nu\},
\]
for some $\nu$ to be chosen.
%\textbf{Note: The norm $\|(\phi,\varphi)\|_{\infty,\nu}=
%\max\{\|\phi\|_\nu, \|\varphi\|_\nu\}$ gives better estimates. We skip $\|\cdot \|_{1,\nu}$.}

In~\cite{BaSe08}, it is proven that the operators $\FF^{\pm}$ are contractive operators in some ball of $\XX_3\times\XX_3$ if $\rr\geq\rr^*$ is big enough. Consequently the existence of solutions $\psi^\pm$ of equation~\eqref{eq:InnerModifiedgen} in the domains $\DDD^{\pm}_\rr$ is guaranteed.  However, we want to be explicit in the estimates to compute the smallest $\rr_*$ such that one can prove that $\FF^{\pm}$ are contractive operators.
%the existence of the invariant manifolds in the domains $\DDD^{\pm}_{\rr_*}$ for the equation~\eqref{eq:InnerModifiedgen}.
%To this end we state two technical results. %

To this end, we need to control the dependence on $\rr$ of the Lipschitz constant of the operators $\mathcal{F}^{\pm}$. 
Let us explain briefly the procedure, which is performed only for the $-$ case being the $+$ case analogous. 
\begin{itemize}
    \item In Section~\ref{sec:step0}, we provide explicit bounds for the norm of the linear operator $\mathcal{B}^-$ in \eqref{def:OperatorB}.
\item In Section~\ref{sec:step1}, we define a set of constants depending on the nonlinear terms $\mathcal{S}$ (see \eqref{def:Rgen})  of the inner equation.
    \item We deal with the bounds of the first iteration, $\mathcal{F}^+(0)$ in Section~\ref{sec:step2}. We conclude that it belongs to a closed ball of $\mathcal{X}_3 \times \mathcal{X}_3$ if $\rho\geq \rho_*^1$ where $\rho_*^1$ is determined by the constants in the previous step. The radius of the ball, $M_0(\rho)/2$ is fully determined also by the previous constants.
    \item In Section~\ref{sec:step3}, we provide explicit bounds of the derivative of the nonlinear operator $\mathcal{S}$ and consequently of its  Lipschitz constant, which depends on $\rho$. These computations hold true for values of $\rho \geq \rho_*^2\geq \rho_1^*$ with $\rho_*^2$ satisfying some explicit conditions. 
    \item In Section~\ref{sec:step4}, for $\rho\geq \rho_*^2$, we compute the Lipschitz constant  $L(\rho)$ of $\mathcal{F}^-$ in the closed ball of $\mathcal{X}_3 \times \mathcal{X}_3$ of radius $M_0(\rho)$.
    \item In Section~\ref{sec:step5}, we set $\rho_*\geq \rho_*^2$ for the existence result. We choose $\rr^*$ such that $L(\rho_*)\leq \frac{1}{2}$. Then, since 
    $$
    \|\psi^-\|_3 \leq \|\mathcal{F}^-(0) \|_3 + 
    \|\mathcal{F}^-(\psi^-) - \mathcal{F}^-(0) \|_3 \leq \frac{M_0(\rho)}{2} + L(\rho) \|\psi^-\|_3 \leq M_0(\rho)  
    $$
    the fixed point theorem ensures the existence of a fixed point $\psi^-$ satisfying $\|\psi^- \|_3 \leq M_0(\rho)$ for $\rr\geq\rr^*$.
    \item Finally, we compute $\Delta \phi(-i\rho_*)$ by computed assisted proofs techniques. This completes the Part 1 of the algorithm since $\Delta \phi(-i\rho_*)\neq 0$ implies $\Theta\neq 0$.
\end{itemize}

All the steps described above are written with all the detailed constants in Section~\ref{sec:AlgorithmMethod1}.

\subsubsection{Rigourous computation of the Stokes constant}\label{sec:theoreticalPart2}
In this section we describe a method to compute rigorously the Stokes constant $\Theta$ defined in \eqref{def:StokesLimit} (Part 2 of the algorithm). 
%the smallest $\rr$ such that the existence of $\psi^{\pm}$ in $\DDD^+_\rr$ is guaranteed.
The method is based in the alternative formula for $\Theta$ proposed in~\eqref{formulaTheta0}: 
\begin{equation}\label{def:Thetabis}
\Theta =  i\frac{e^{\al\left (\rr -i\h\log\rr-\h \frac{\pi}{2}\right )}\Delta \phi(-i\rr)}{\rr} (1+ g(\rho)), \qquad \lim_{\rho \to \infty} g(\rho)=0.
\end{equation}
Let us to explain how this formula is derived. The key point is to analyze the difference
\[
 \Delta\psi(s)=\psi^+(s)-\psi^-(s)
\]
as a solution of a linear equation on the vertical axis $\Im s\in (-\infty, -\rr_*)$ where $\rr_*$ is provided by the method explained in Section~\ref{sec:theoreticalPart1}.
Indeed $\Delta\psi=(\Delta \phi, \Delta\varphi)$ satisfies the equation
\[
  \begin{pmatrix}\Delta\phi'\\\Delta\varphi'\end{pmatrix}=
  \left(\AAA(s)+\KK(s)\right)\begin{pmatrix}\Delta\phi\\\Delta\varphi\end{pmatrix}
\]
where
\begin{equation}\label{def:K}
\KK(s)=\int_0^1D\SSS\left(\psi^-(s)+t(\psi^+(s)-\psi^-(s)),s\right)dt
\end{equation}
and $\SSS$ is given in \eqref{def:Rgen}. We look for the linear terms of lower order in $s^{-1}$  of $\SSS$. Indeed, we have that
$$
\SSS(\psi,s)= \frac{1}{s} \left (\begin{array}{c} -\al i \h  \phi \\ \al i\h \varphi \end{array}\right )
+\widetilde{\SSS}(\psi,s)
$$
with $\widetilde{\SSS}(\psi,s)=\mathcal{O}(|s|^{-2})$ when
$\psi\in \XX_3$ (see \eqref{def:BanachChi}). Then,
$\Delta\psi$ satisfies the equation
\begin{equation}\label{def:eqdif0}
  \begin{pmatrix}\Delta\phi'\\\Delta\varphi'\end{pmatrix}=
  \left(\widetilde{\AAA}(s)+\widetilde{\KK}(s)\right)\begin{pmatrix}\Delta\phi\\\Delta\varphi\end{pmatrix}
\end{equation}
where
\begin{equation}\label{def:MatrixtildeA}
\widetilde{\AAA}(s) =
 \begin{pmatrix}
-i\al+\frac{1}{s} -i\al \frac{\h}{s}&0\\
0&i\al+\frac{1}{s} +i\al \frac{\h}{s}
 \end{pmatrix}
\end{equation}
and
\begin{equation}\label{def:Ktilde}
\widetilde{\KK}(s)=\int_0^1D\widetilde{\SSS}\left(\psi^-(s)+t(\psi^+(s)-\psi^-(s)),s\right)dt.
\end{equation}
A fundamental matrix for the linear system $z'=\widetilde{A}(s) z$ is
$$
\begin{pmatrix} s e^{-i\al( s+\h \log s)} & 0 \\ 0 & s e^{i\al( s+\h \log s)} \end{pmatrix}.
$$
Therefore, any solution of  system~\eqref{def:eqdif0} can be expressed as
$$
\begin{pmatrix}\Delta\phi\\\Delta\varphi\end{pmatrix} = \begin{pmatrix}
\dps se^{-i\al (s+\h \log s)} \left [\kk_0 + \int_{-i\rho}^s\frac{e^{i\al ( t+\h \log t)}}{t}\left( \widetilde{\KK}_{11}\Delta\phi+\widetilde{\KK}_{12}\Delta\varphi\right)dt \right ]\\
       \dps se^{i\al (s+\h \log s)}
       \left [ \kk_1 + \int_{-i\rho}^s\frac{e^{-i\al ( t+\h \log t)}}{t}\left(\widetilde{ \KK}_{21}\Delta\phi+\widetilde{\KK}_{22}\Delta\varphi\right)dt
        \right ]\end{pmatrix}
$$
with $\kk_0,\kk_1$ two constants.

Since $|\psi^{\pm}|\leq M_0(\rho) |s|^{-3}$,  $\Delta \psi $ goes to $0$ as $\Im s\to -\infty$ and, therefore,
$$
\kk_1=-\int_{-i\rho }^{-i\infty} \frac{e^{-i\al ( t+\h \log t)}}{t}\left(\widetilde{ \KK}_{21}\Delta\phi+\widetilde{\KK}_{22}\Delta\varphi\right)dt.
$$
Then, we deduce that the difference $\Delta\psi(s)$ is a fixed point of the equation
\begin{equation}\label{eq:diferenciapuntfix}
\begin{split}
\Delta\psi(s) &= \Delta\psi^0(s)+\GG(\Delta \psi) (s), \ \mbox{
where} \\
\Delta\psi^0(s)&=\left(\begin{array}{c} s e^{-i\al( s+\h \log s)}\kappa_0\\0\end{array}\right)
\end{split}
\end{equation}
with $\kappa_0$ a constant depending on $\rho$ and $\GG$ is the linear operator
\begin{equation}\label{def:OperatorG}
 \GG(\Delta \psi)=\begin{pmatrix}
       \dps se^{-i\al (s+\h \log s)}\int_{-i\rr}^s\frac{e^{i\al ( t+\h \log t)}}{t}\left( \widetilde{\KK}_{11}\Delta\phi+\widetilde{\KK}_{12}\Delta\varphi\right)dt\\
       \dps se^{i\al (s+\h \log s)}\int_{-i\infty}^s\frac{e^{-i\al ( t+\h \log t)}}{t}\left(\widetilde{ \KK}_{21}\Delta\phi+\widetilde{\KK}_{22}\Delta\varphi\right)dt
        \end{pmatrix}.
\end{equation}
By construction, $\kk_0$ is defined as
\begin{equation}\label{def:kapparho}
 \kk_0=\kk_0(\rr)=i\frac{e^{\al\left (\rr -i\h\log\rr-\h \frac{\pi}{2}\right )}\Delta \phi(-i\rr)}{\rr}.
\end{equation}
Using \eqref{def:StokesLimit} and \eqref{eq:diferenciapuntfix}, we have
\begin{equation}\label{def:Formula2Stokes}
\begin{aligned} 
 \Theta & =\lim_{\Im s\to-\infty}s^{-1}e^{i\al( s + \h \log s)}\Delta\phi(s)\\ 
 &=\kk_0\left(1+\kk_0^{-1}\lim_{\Im s\to-\infty}s^{-1}e^{i\al (s + \h \log s)}\GG_1\left(\Delta\psi(s)\right)\right).
 \end{aligned}
\end{equation}
We use equality~\eqref{def:Formula2Stokes} to obtain formula~\eqref{def:Thetabis} of $\Theta$. 
To bound $|g(\rho)|$ in formula ~\eqref{def:Thetabis},
%of $\Theta$
%To this end 
we need to control the linear operator $s^{-1}e^{i\al (s + \h \log s)}\GG_1$. To this end, we consider a  norm with exponential weights,
\begin{equation}\label{def:nomexp}
 \|\psi\|=\max\left\{\max_{s\in E} \left|s^{-1}e^{i\al (s+\h \log s)}\phi\right|,\max_{s\in E} \left|e^{i\al ( s+\h \log s) }\varphi\right|\right\}
\end{equation}
with $E=\{\Re s=0, \, \Im s\in (-\rr_*,-\infty)\}$.
 
Observe that~\eqref{eq:diferenciapuntfix} can be rewritten
\[
(\text{Id}-\GG)(\Delta\psi)=\Delta\psi^0.
\] 
Now we see that $\text{Id}-\GG$ is an invertible operator. Indeed, in~\cite{BaSe08}, it was proven that for $\rho $ big enough
\begin{equation}\label{def:Arho}
\|\mathcal{G}_1(\Phi)\|\leq A_1(\rho) \|\Phi \|, \qquad \|\mathcal{G}_2(\Phi)\|\leq A_2(\rho)\|\Phi\|
\end{equation}
with $0<A(\rho):=\max\{A_1(\rho),A_2(\rho)\}<1$. Moreover, \cite{BaSe08} also shows that
$$
\lim_{\rho \to \infty} A(\rho)=0.
$$
These estimates imply that, if $\rho$ is big enough,  $\text{Id}-\GG$ is invertible and therefore $\Delta\psi=(\text{Id}-\GG)^{-1}(\Delta\psi^0)$. Moreover,
\begin{equation}\label{def:Estkappa0}
 \|\Delta\psi\|\leq \frac{1}{1-A(\rho)}\|\Delta\psi^0\|= \frac{|\kk_0(\rr)|}{1-A(\rho)}
\end{equation}
and this inequality directly gives
%\[
%|\Delta \phi(s)|\le \frac{|\kk_0(\rr)|}{1-A(\rho)} |s| 
%\left |e^{-i\al (s + h_0 \log s)} \right | \leq \frac{|\kk_0(\rr)|}{1-A(\rho)}
%|s|
%e^{\al \Im (s)} e^{-\al \h \pi /2 }
%\]
$$
\left |\lim_{\Im s \to -\infty} s^{-1} e^{i\al(s+h_0\log s)} \mathcal{G}_1(\Delta \psi(s)) \right |\leq    A_1(\rho) \frac{|\kk_0(\rr)|}{1-A(\rho)}.
$$
Therefore, from~\eqref{def:Formula2Stokes}, we can conclude that the Stokes constant $\Theta$, which is independent of $\rr$, can be computed as 
\[
 \Theta=\kk_0(\rr) (1+g(\rr)),
\]
for any $\rr $ big enough, where $\kappa_0$ is given in \eqref{def:kapparho} and $g$ satisfies
\begin{equation}\label{def:Mbarrho}
 |g(\rr)|\leq  \overline{M}(\rho):=\frac{A_1(\rho) }{1- A(\rr)}.
\end{equation}
Since $A(\rho),A_1(\rho)$ go to zero as $\rho \to \infty$, the same happens for $\overline{M}(\rr)$. Then~\eqref{def:Thetabis} is proven. 

Notice that the relative error to approximate $\Theta$ by $\kappa_0$ is 
$$
\frac{
\left |\Theta - \kappa_0(\rho)\right | }{|\kappa_0(\rho)|} \leq \overline{M}(\rho).
$$
%where $\overline{M}(\rho)$ goes to $0$ as $\rho$ goes to $\infty$. 
As a consequence, 
%for $\rho$ big enough 
$$
|\Theta | \in \left [|\kappa_0(\rho)| (1-\overline{M}(\rho)), |\kappa_0(\rho)| (1+\overline{M}(\rho))\right ].
$$
%That means that, taking $\rho$ big enough we can compute $\Theta$ with a high accuracy provided that $\kappa_0(\rho)$ 
In Section~\ref{sec:method1} the procedure described above is implemented:  
\begin{itemize}
    \item Following the fixed point argument in~\cite{BaSe08}, in Section~\ref{sec:step6} we give a explicit formula for $A(\rho)=\max\{A_1(\rho),A_2(\rho)\}$ in~\eqref{def:Arho} for $\rr\geq \rr_*$, where $\rr^*$ is the constant given by Part 1.
\item In Section~\ref{sec:step7}, we set $\rho_0\geq \rho_*$ such that $\overline{M}(\rho)<1$ for $\rho \geq \rho_0$. 
 \end{itemize}

%In Section~\ref{sec:method2} we give a method to compute an explicit value for the constant $C$.

\subsection{Computing the Stokes constant: method}\label{sec:method1}
%The theoretical framework}
In this section we are going to give explicit expressions for all the constant involved in the method explained in the previous section.
\subsubsection{The linear operator $\mathcal{B}^{-}$}\label{sec:step0}
\begin{lemma}\label{cor:linearB}
Consider the linear operator $\B^-$ defined
in~\eqref{def:OperatorB}.
\begin{enumerate}
\item When $\nu>1$, the linear operator $\B^-:\XX_{\nu}\times \XX_{\nu}\to \XX_{\nu-1}\times \XX_{\nu-1}$ is continuous and
$$
\|\B^-(\psi)\|_{\nu-1} \leq B_{\nu+1} \|\psi\|_\nu
$$
where
\begin{equation}\label{def:Bnu}
\begin{aligned}
B_{m}&= \frac{\pi}{2} \frac{(m-3))!!}{(m-2)!!} \qquad &&\text{ if $m$ is even}\\
B_m&= \frac{(m-3)!!}{(m-2)!!} \qquad &&\text{ if $m$ is odd}.
\end{aligned}
\end{equation}
\item When $\nu>0$, the linear operator $\B^-:\XX_{\nu}\times \XX_{\nu}\to \XX_{\nu}\times \XX_{\nu}$ is continuous and, for all $0<\gamma \leq  \beta$ (see \eqref{def:domainInnner}),
$$
\|\B^-(\psi)\|_\nu \leq \frac{1}{\al \sin \gamma (\cos \gamma)^{\nu+1}}\|\psi\|_\nu.
$$
Define
$\gamma_*\in (0,\frac{\pi}{2})$ such that $\sin^2 \gamma_* =\frac{1}{\nu+2}$. If $\gamma_* \leq \beta$, %for this value of $\gamma_*$,
\begin{equation}\label{def:Cnu}
\|\B^-(\psi)\|_\nu \leq C_{\nu} \|\psi\|_\nu\qquad \text{ where }\qquad C_{\nu}=\frac{(\nu+2)^{\frac{\nu+2}{2}}}{\al (\nu+1)^{\frac{\nu+1}{2}}}.
\end{equation}
\end{enumerate}
\end{lemma}
This lemma is proven in Appendix \ref{app:ProofOperator}.
%\subsection{Computing the Stokes constant: Explicit estimates}\label{subsec:paramgen}

From now on we choose $\beta$, the angle in the definition~\eqref{def:domainInnner} of $\DDD^-_{\rr}$, be such that
$6 \sin \beta^2=1$. Then for all $\nu\geq 4$, the optimal value
$\gamma_*$ in second item of Lemma~\ref{cor:linearB}
satisfies that
$\gamma_*\leq \beta$ and the optimal bound \eqref{def:Cnu} will be used throughout the paper.

%Assume then that we are dealing with a concrete example of some inner equation~\eqref{eq:InnerModifiedgen}. 
We emphasize that, if $s\in \DDD^{-}_\rr$, one has that $|s|\geq \rr$.
Recall that we are looking for $\rr_*$ the minimum value for $\rr$ to ensure that the inner equation has a solution $\psi^-$ defined in $\DDD^-_\rr$. Since we need $\rr^{-1}_*$ to be small, we start by assuming that $\rr_*\geq 2$. We will change this value along the proof.

%\textcolor{blue}{Now we prove Theorem \ref{th:inner} giving explicit values for all the estimates %needed in the proof. This allows us to give an algorithm for Part 1 which is explained in Section \ref{sec:AlgorithmMethod1}.}
%We split the algorithm in some steps below.

\subsubsection{Explicit constants for the inner equation}\label{sec:step1}
We consider the $\max$ norm $|(x,y,z)|=\max\{|x|,|y|,|z|\}$. 
Let $a_3=\lim _{z\to 0}z^{-3}F_1(0,0,z)$, $\h=\lim _{z\to 0}z^{-3}H(0,0,z)$ and $C_{F}^0,C_{H}^0,\overline{C}_H^0$ be such that for $|z|\leq \frac{1}{2}$,
\begin{equation}\label{def:estiamatesFH}
\begin{split}
\left |\Delta F_1(z)\right |= \left |F_1(0,0,z)+a_3z^3\right |\leq C_{F}^0|z|^4, \\
\left |\Delta F_2(z)\right |=\left |F_2(0,0,z)+\overline{a_3}z^3\right |\leq C_{F}^0|z|^4,\\
\left |\Delta H(z)\right |=\left |H(0,0,z)+\h z^3\right |\leq C_{H}^0|z|^4,\\
\left |H(0,0,z)\right |\leq \overline{C}_H^0 |z|^3.
\end{split}
\end{equation}
We also introduce $C_F, C_F^{\phi,\varphi}$, $C_H,C_{H}^{\phi,\varphi}$ such that, for
$|(x,y)|\leq |z|$,
\begin{equation}\label{constants2}
\begin{split}
\left | H(x,y,z)\right |&\leq C_H|(x,y,z)|^3 \leq C_H|z|^3  \\
\left | F_{1,2}(x,y,z)\right |&\leq C_F|(x,y,z)|^3 \leq C_F|z|^3, \\
\left |\partial_x F_{1,2}(x,y,z) \right | &\leq C_F^{\phi}  |(x,y,z)|^2 \leq C_F^{\phi}|z|^2, \\
\left |\partial_y F_{1,2}(x,y,z) \right | &\leq C_F^{\varphi}  |(x,y,z)|^2 \leq C_F^{\varphi}|z|^2, \\
\left |\partial_x H(x,y,z) \right | &\leq C_H^{\phi}  |(x,y,z)|^2 \leq C_H^{\phi}|z|^2, \\
\left |\partial_yH(x,y,z) \right | &\leq C_H^{\varphi}  |(x,y,z)|^2 \leq C_H^{\varphi}|z|^2 .
\end{split}
\end{equation}
As a consequence, setting 
\begin{equation}\label{def:Cbar}
\overline{C}_H= C_H^\phi + C_H^\varphi \quad \text{ and }\quad \overline{C}_F= C_F^\phi + C_F^\varphi
\end{equation}
we have
\[
\begin{split}
\left |H(x,y,z) + \h z^3\right | &\leq C_H^0 |z|^4 + \overline{C}_H |(x,y)| |z|^2\\
\left |F_1(x,y,z) +a_3z^3\right | &\leq C_F^0 |z|^4 + \overline{C}_F |(x,y)| |z|^2\\
\left |F_2(x,y,z) +\overline{a_3}z^3\right |& \leq C_F^0 |z|^4 + \overline{C}_F |(x,y)| |z|^2.
\end{split}
\]
%\textcolor{blue}{Check that the same constant is true for $F_1$ and $F_2$}

\subsubsection{Bounds for the norm of the first iteration}\label{sec:step2}
The second step in the proof consists on studying   
%\psi_0(s)=(\phi_0(s),\varphi_0(s))=
%\begin{equation}\label{def:psi0minus}%
%\psi_0^-(s)=
$\FF^-(0)(s)=\B^- (\SSS(0,s))$, 
%\end{equation}
where $\FF^-$ is the operator introduced in \eqref{def:F}.
\begin{lemma}\label{lem:firstiterationgen}
Chose any  $\rr_*^1 >\max\{2, \overline{C}_H^0\}$, take  $\rr \geq \rr_*^1$ and define 
$$
\mathcal{C}_0(\rr)= \frac{C_{F}^0 + |a_3|\overline{C}_H^0}{1- \frac{|\overline{C}_H^0|}{\rr}}.
$$
Then
$\FF^-(0)\in \XX_3\times\XX_3$ and
$$\|\FF^-(0)\|_3 \leq \frac{11 |a_3|}{3\al}+ B_5 \mathcal{C}_0(\rr).$$
\end{lemma}
\begin{proof}
By \eqref{def:Rgen}, it is clear that
\[
\SSS_1(0,s) - \frac{a_3}{s^3}= \frac{F_1(0,0,-s^{-1})}{1+ s^2 H(0,0,-s^{-1})}-\frac{a_3}{s^3}= \frac{\Delta F_1 (-s^{-1})-\frac{a_3 H(0,0,-s^{-1})}{s}}{1+s^2 H(0,0,-s^{-1})}
\]
and therefore
$$
\left | \SSS_1(0,s)-\frac{a_3}{s^3}\right | \leq \frac{1}{|s|^4} \frac{C_{F}^0 + |a_3|\overline{C}_H^0}{1- \frac{\overline{C}_H^0}{ \rr}}.
$$
An analogous bound works for $ \SSS_2(0,s)$ and therefore
\begin{equation}\label{firstSSS}
\| \SSS(0,s)- s^{-3} (a_3,\overline{a_3})\|_4 \leq
\frac{C_{F}^0 + |a_3|\overline{C}_H^0}{1- \frac{|\overline{C}_H^0|}{\rr}}=\mathcal{C}_0(\rr).
\end{equation}
We introduce $\SSS_0(s)= s^{-3} (a_3,\overline{a_3})$.
We have that
\begin{equation}\label{B1_aprox1}
\B_1^-(\SSS_0(s)) = a_3 s \int_{-\infty}^0 \frac{e^{i\al t}}{(s+t)^4} \, dt = \frac{a_3}{i \al s^3 } + \frac{4sa_3}{i\al}\int_{-\infty}^0 \frac{e^{i\al t}}{(s+t)^5 }\, dt .
\end{equation}
Notice that, for $s\in \DDD^-_\rr$ and $t\in \R$,
$|s+t|^2 \geq |s|^2 + t^2$. Then, using also Lemma~\ref{lem:Im} (see Appendix \ref{app:ProofOperator}),
\[
\left |s \int_{-\infty}^0 \frac{e^{i\al t}}{(s+t)^5 }\, dt
\right | \leq \frac{1}{|s|^3} \int_{-\infty}^0 \frac{1}{\big (t^2 + 1\big )^{5/2}}  = \frac{2}{3 |s|^3}.
\]
Using this last bound and formula~\eqref{B1_aprox1} we obtain
$$
\left |\B_1^-(\SSS_0(s)) \right |\leq  \frac{1}{|s|^3} \left (\frac{|a_3|}{\al} + \frac{8|a_3|}{3 \al}\right ) \leq \frac{11 |a_3|}{3 \al|s|^3}.
$$
To finish we notice that, from~\eqref{firstSSS} and the first item of Corollary~\ref{cor:linearB},
$$
\left \|\FF^-_1(0)\|_3\leq \| \B^-_1(\SSS_0)\right \|_3 + \left \|\B^-_1 (\SSS(0,\cdot)- \SSS_0)\right \|_3\leq \frac{11|a_3|}{3\al}+\mathcal{C}_0(\rr) B_{5}.
$$
Analogous computations lead to the same estimate for $\|\FF_2^-(0)\|_3$.
%$$
%\|\B^-_2(\SSS_0)\|_3 \leq \frac{11 |a_3|}{3\al}.
%$$

%where in the last inequality we have used .
\end{proof}

\subsubsection{The Lipschitz constant of  $\SSS$}\label{sec:step3}
Let 
\begin{equation}\label{def:M0}
M_0(\rr)= \frac{22 |a_3|}{3\al} + 2 B_5 \mathcal{C}_0(\rr)
\end{equation}
in such a way that $2 \|\FF^-(0)\|_3 \leq M_0(\rr)$.

\begin{lemma}\label{lem:DSgen}
Assume that $\|\phi\|_3,\|\varphi\|_3 \leq M_0(\rr) $ and take $\rr\geq  \rr_*^2$ being $\rr_*^2\geq \rr_*^1$ such that
\begin{equation}\label{def:minM0}
\min \left \{ 1- \frac{b M_0^2(\rr_*^2)}{(\rr_*^2)^4} - \frac{C_H}{(\rr_*^2)}, 1 - \frac{M_0(\rr_*^2)}{(\rr_*^2)^2}\right \}
>0.
\end{equation}
Then
\[
\begin{split}
\big |\partial_\phi \SSS_1(\psi,s) \big | ,
\big |\partial_\varphi \SSS_2(\psi,s) \big | &\leq \frac{M_{11}^1(\rr)}{|s|} + \frac{M_{11}^2(\rr)}{|s|^2}+ \frac{M_{11}^3(\rr)}{|s|^3}+ \frac{M_{11}^4(\rr)}{|s|^4}
\\
\big |\partial_\varphi \SSS_1(\psi,s)|, |\partial_\phi\SSS_2(\psi,s) \big | &\leq \frac{M_{12}^2(\rr)}{|s|^2}+ \frac{M_{12}^3(\rr)}{|s|^3}+ \frac{M_{12}^4(\rr)}{|s|^4}
\end{split}
\]
with
\begin{equation}\label{def:Mij}
\begin{split}
M_{11}^1(\rr) =& \frac{\al |\h|}{1 - \frac{bM_0^2}{\rr^4}-\frac{C_H}{\rr}}, \\
M_{11}^2(\rr)  =& \frac{|\h| + \al C_H^0+ C_{F}^\phi}{1 - \frac{bM_0^2}{\rr^4}-\frac{C_H}{\rr}}, \\
M_{11}^3(\rr)  =& \frac{1 }{ \left (1 - \frac{bM_0^2}{\rr^4} - \frac{C_H}{\rr}\right )^2} \left [M_0\al C_{H}^\phi +C_F C_H^\phi + \left (\al \overline{C}_H M_0 + C_H^0 \right )
\left ( 1 -\frac{b M_0^2}{\rr^4} - \frac{C_H}{\rr}\right )\right ], \\
M_{11}^4 (\rr) =& \frac{M_0}{ \left (1 - \frac{bM_0^2}{\rr^4} - \frac{C_H}{\rr}\right )^2}\\
&\cdot
\left [ b (C_F + \al M_0) + C_H^\phi
+ \left (\al b M_0 +\overline{C}_H+\frac{bM_0}{\rho}\right )\left (1 -\frac{b M_0^2}{\rr^4} - \frac{C_H}{\rr}  \right ) + \frac{bM_0}{\rr}
\right]
\\
M_{12}^2(\rr)  =& \frac{C_{F}^\varphi}{1 - \frac{bM_0^2}{\rr^4}-\frac{C_H}{\rr}}, \\
M_{12}^3(\rr)  =& \frac{M_0\al C_{H}^\varphi +C_F C_H^\varphi }{ \left (1 - \frac{bM_0^2}{\rr^4} - \frac{C_H}{\rr}\right )^2}, \\
M_{12}^4 (\rr) =& \frac{M_0}{ \left (1 - \frac{bM_0^2}{\rr^4} - \frac{C_H}{\rr}\right )^2}
\left ( b (C_F + \al M_0) + C_H^\varphi
+ \frac{bM_0}{\rr}
\right ).
\end{split}
\end{equation}
\end{lemma}
\begin{proof}
Notice that $\rr_*^2\geq \rr_*^1$ and therefore, Lemma~\ref{lem:firstiterationgen} can be applied for $\rr \geq \rr_*^2$. Moreover, if $s\in \DDD^-_\rr$,
$$
|\psi(s)| \leq  \frac{M_0(\rr)}{|s|^3} \leq
\frac{1}{|s|}
$$
so that the bounds in~\eqref{constants2} can also be used.

We start with $\partial_{\phi}\SSS$. We introduce
\[
\begin{split}
S_1(\psi,s)&=\frac{\partial_{\phi} F_1(\psi,-s^{-1})}{1+ s^2 \big (b\phi\varphi+ H(\psi,-s^{-1})\big )} +
\frac{F_1(\psi,-s^{-1}) s^2 \big (b \varphi + \partial_{\phi} H(\psi,-s^{-1})\big )}{\left (1+ s^2 \big (b\phi \varphi+ H(\psi,-s^{-1})\big )\right )^2 }
\\
S_2(\psi,s)&= \frac{\left (\al i -\frac{1}{s}\right )s^2 \big (b\phi\varphi+ H(\psi,-s^{-1})\big )}{1 + s^2 \big (b\phi\varphi+ H(\psi,-s^{-1})\big )}- \frac{\left (\al i -\frac{1}{s}\right ) \phi s^2 \big (b \varphi + \partial_{\phi} H(\psi,-s^{-1})\big )}{\left (1+ s^2 \big (b\phi\varphi+ H(\psi,-s^{-1})\big )\right )^2 }.
\end{split}
\]
Straightforward computations, lead us to
$$
\partial_\phi \SSS_1 = S_1 + S_2.
$$
When $\|\phi\|_3,\|\varphi\|_3 \leq M_0(\rr)$,
\[
\begin{split}
|S_1(\psi,s)|\leq&\frac{1}{|s|^2 } \frac{C_F^\phi}{1 - \frac{bM_0^2}{\rr^4} - \frac{C_H}{\rr}}
+\frac{1}{|s|^3}\frac{C_F  \left (C_H^\phi + b \frac{M_0}{|s|}\right )}{\left (1 - \frac{bM_0^2}{\rr^4} - \frac{C_H}{\rr}\right )^2} \\
\leq & \frac{1}{|s|^2 } \frac{C_F^\phi}{1 - \frac{bM_0^2}{\rr^4} - \frac{C_H}{\rr}} +
\frac{1}{|s|^3} \frac{C_F C_H^\phi}{\left (1 - \frac{bM_0^2}{\rr^4} - \frac{C_H}{\rr}\right )^2}
+ \frac{1}{|s|^4} \frac{bM_0 C_F}{\left (1 - \frac{bM_0^2}{\rr^4} - \frac{C_H}{\rr}\right )^2}, \\
|S_2(\psi,s)|  \leq &\frac{1}{|s|}\frac{\left (\al + \frac{1}{|s|} \right )
\left (|\h| + \frac{C_H^0}{|s|} + \frac{\overline{C}_HM_0}{|s|^2}+ b \frac{M_0^2}{|s|^3} \right) }{1 - \frac{bM_0^2}{\rr^4}-\frac{C_H}{\rr}} +
\frac{M_0}{|s|}\frac{\left (\al + \frac{1}{|s|} \right ) \left  (\frac{C_H^\phi}{|s|^2} +  \frac{bM_0}{|s|^3}  \right)}{ \left (1 - \frac{bM_0^2}{\rr^4} - \frac{C_H}{\rr}\right )^2}\\
= & \frac{1}{1 - \frac{bM_0^2}{\rr^4}-\frac{C_H}{\rr}} \left (\frac{\al |\h|}{|s|} + \frac{|\h| + \al C_H^0}{|s|^2} +  \frac{\al \overline{C}_H M_0 + C_H^0}{|s|^3}+
\frac{\al b M_0^2 + \overline{C}_H M_0}{|s|^4} + \frac{bM_0^2}{|s|^5}\right )
\\ &+ \frac{1}{ \left (1 - \frac{bM_0^2}{\rr^4} - \frac{C_H}{\rr}\right )^2}
\left(\frac{M_0 \al C_H^\phi}{|s|^3} + \frac{\al b M_0^2 + M_0 C_{H}^\phi }{|s|^4} +
\frac{b M_0^2}{|s|^5}\right ).
\end{split}
\]
Therefore we have that
$$
\big |\partial_\phi \SSS_1(\psi,s) \big | \leq \frac{M_{11}^1(\rr)}{|s|} + \frac{M_{11}^2(\rr)}{|s|^2}+ \frac{M_{11}^3(\rr)}{|s|^3}+ \frac{M_{11}^4(\rr)}{|s|^4}
$$
where $M_{11}^k$ are the constants introduced in the lemma.

%with
%\[
%\begin{split}
%M_{11}^3(\rr)  =& \frac{1 }{ \left (1 - \frac{bM_0^2}{\rr^4} - \frac{C_H}{\rr}\right )^2} \left (M_0\al C_{H}^\phi +C_F C_H^\phi + \left (\al \overline{C}_H M_0 + C_H^0 \right )
%\left ( 1 -\frac{b M_0^2}{\rr^4} - \frac{C_H}{\rr}\right )\right ) \\
%M_{11}^4 (\rr) =& \frac{M_0}{ \left (1 - \frac{bM_0^2}{\rr^4} - \frac{C_H}{\rr}\right )^2}\\
%&\cdot\left [ b (C_F + \al M_0) + C_H^\phi
%+ \left (\al b M_0 +\overline{C}_H+\frac{bM_0}{\rho}\right )\left (1 -\frac{b M_0^2}{\rr^4} - \frac{C_H}{\rr}  \right ) + \frac{bM_0}{\rr}
%\right ].
%\end{split}
%\]
We now compute a bound for $\partial_{\varphi} \SSS$. As for
$\partial_{\phi}\SSS$, we define
\[
\begin{split}
S_1(\psi,s)&=\frac{\partial_{\varphi} F_1(\psi,-s^{-1})}{1+ s^2 \big (b\phi\varphi+ H(\psi,-s^{-1})\big )} +
\frac{F_1(\psi,-s^{-1}) s^2 \big (b \phi + \partial_{\varphi} H(\psi,-s^{-1})\big )}{\left (1+ s^2 \big (b\phi \varphi+ H(\psi,-s^{-1})\big )\right )^2 }
\\
S_2(\psi,s)&=  -\frac{\left (\al i -\frac{1}{s}\right ) \phi s^2 \big (b \phi + \partial_{\varphi} H(\psi,-s^{-1})\big )}{\left (1+ s^2 \big (b\phi\varphi+ H(\psi,-s^{-1})\big )\right )^2 }.
\end{split}
\]
and we notice that
$$
\partial_\varphi \SSS_1 = S_1 + S_2.
$$
We have that, if
$\|\phi\|_3,\|\varphi\|_3 \leq M_0(\rr)$,
\[
\begin{split}
|S_1(\psi,s)|
\leq & \frac{1}{|s|^2 } \frac{C_F^\varphi}{1 - \frac{bM_0^2}{\rr^4} - \frac{C_H}{\rr}} +
\frac{1}{|s|^3} \frac{C_F C_H^\varphi}{\left (1 - \frac{bM_0^2}{\rr^4} - \frac{C_H}{\rr}\right )^2}
+ \frac{1}{|s|^4} \frac{bM_0 C_F}{\left (1 - \frac{bM_0^2}{\rr^4} - \frac{C_H}{\rr}\right )^2}, \\
|S_2(\psi,s)|  \leq &\frac{1}{ \left (1 - \frac{bM_0^2}{\rr^4} - \frac{C_H}{\rr}\right )^2}
\left(\frac{M_0 \al C_H^\varphi}{|s|^3} + \frac{\al b M_0^2 + M_0 C_{H}^\varphi }{|s|^4} +
\frac{b M_0^2}{|s|^5}\right ).
\end{split}
\]
Then
$$
\big |\partial_\varphi \SSS_1(\psi,s) \big | \leq \frac{M_{12}^2(\rr)}{|s|^2}+ \frac{M_{12}^3(\rr)}{|s|^3}+ \frac{M_{12}^4(\rr)}{|s|^4}
$$
with the constants $M_{12}^k$ defined in the lemma.

Since the bounds for $F_1,\partial_\phi F_1,\partial_\varphi F_1$ are the same as for $F_2,\partial_\phi F_2,\partial_\varphi F_2$ and using the symmetry in the definition of $\SSS$, we have that
\[
\begin{split}
|\partial_\phi \SSS_2(\psi,s)| & \leq \frac{M_{12}^2(\rr)}{|s|^2}+ \frac{M_{12}^3(\rr)}{|s|^3}+ \frac{M_{12}^4(\rr)}{|s|^4} \\
\big |\partial_\varphi \SSS_2(\psi,s) \big | &\leq \frac{M_{11}^1(\rr)}{|s|} + \frac{M_{11}^2(\rr)}{|s|^2}+ \frac{M_{11}^3(\rr)}{|s|^3}+ \frac{M_{11}^4(\rr)}{|s|^4}.
\end{split}
\]
\end{proof}

As a corollary we obtain the following.
\begin{corollary}\label{coro:Sbounds} If $\psi,\psi'\in B(M_0(\rr))$ with $\rr\geq \rr_*^2$ as in Lemma~\ref{lem:DSgen}.
Then, there exist functions $\Delta \SSS_j$, $j=1\ldots4$, such that
\[
\SSS(\psi,s)-\SSS(\psi',s) = \sum_{j=1}^4 \Delta \SSS_j(\psi,s)- \Delta \SSS_j(\psi',s)
\]
and
\[
\begin{split}
\|\Delta \SSS_1(\psi,s)-\Delta \SSS_1(\psi',s) \|_4 &\leq
M_{11}^1 (\rho) \|\psi- \psi'\|_3
\\
\|\Delta \SSS_j(\psi,s)-\Delta \SSS_j(\psi',s) \|_{3+j} &\leq
(M_{11}^j (\rr) + M_{12}^j(\rr))\|\psi- \psi'\|_3.
\end{split}
\]
As a consequence
\[
\begin{split}
|\SSS(\psi,s)-\SSS(\psi',s)| \leq& |\psi(s) - \psi'(s)|
\left (\frac{M_{11}^1(\rr)}{|s|} + \frac{M_{11}^2(\rr)+ M_{12}^2(\rr)}{|s|^2} \right .\\&+  \left .\frac{M_{11}^3(\rr)+ M_{12}^3(\rr)}{|s|^3}+ \frac{M_{11}^4(\rr)+M_{12}^4(\rr)}{|s|^4}
\right ).
\end{split}
\]
\end{corollary}
%\textcolor{red}{No estan clares les notacions}
\begin{proof} Indeed:
\[
\begin{split}
\SSS_1(\psi,s)-\SSS_1(\psi',s)=&
(\phi - \phi') \int_{0}^1 \partial_{\phi} \SSS_1 (\psi' + \lambda (\psi-\psi')) \, d\lambda
\\ &+ (\varphi - \varphi') \int_{0}^1 \partial_{\varphi} \SSS_1 (\psi' + \lambda (\psi-\psi')) \, d\lambda \\
=&
(\phi- \phi') (S_1^\phi+ S_2^\phi + S_3^\phi + S_4^\phi)
+ (\varphi- \varphi') (S_2^\varphi + S_3^\varphi + S_4^\varphi)
\end{split}
\]
with $S_{j}^{\phi,\varphi}\in \XX_j$ and
$$
\|S_{j}^{\phi}\|_j \leq M_{11}^j(\rr), \qquad \|S_{j}^{\varphi}\|_j \leq M_{12}^j(\rr).
$$
In analogous way we decompose  $\SSS_2(\psi)-\SSS_2(\psi')$
and by symmetry we obtain that
\begin{equation}\label{difSgen}
\SSS(\psi,s)-\SSS(\psi',s)=
\left (\begin{array}{cc}
S_{11}(s) & S_{12}(s) \\ S_{21}(s) & S_{22}(s) \end{array}\right ) \left (\begin{array}{c} \phi- \phi' \\ \varphi- \varphi'
\end{array}\right )
\end{equation}
with
\[
\begin{split}
|S_{11}(s)|, |S_{22}(s)|& \leq  \frac{M_{11}^1(\rr)}{|s|} + \frac{M_{11}^2(\rr)}{|s|^2}+ \frac{M_{11}^3(\rr)}{|s|^3}+ \frac{M_{11}^4(\rr)}{|s|^4} \\
|S_{12}(s)|, |S_{21}(s)| & \leq \frac{M_{12}^2(\rr)}{|s|^2}+ \frac{M_{12}^3(\rr)}{|s|^3}+ \frac{M_{12}^4(\rr)}{|s|^4}.
\end{split}
\]
Namely, $S_{ij}$ can be decomposed as a sum of functions belonging to the adequate $\XX_{k}$.
Therefore, taking the supremmum norm in~\eqref{difSgen} we get the result.
\end{proof}

\begin{remark}
Notice that, for some concrete functions $F_{1,2}$ and $H$ the general bounds in  that we have used for them and their derivatives (see~\eqref{constants2}) may not be sharp.
%Indeed, the terms of order $|(\phi,\varphi)|$ are bounded
%by $|s|^{-1}$ which is, clearly a rough estimate.
To improve the estimates for $|\partial_{\phi,\varphi} \mathcal{S}|$ in Lemma \ref{lem:DSgen}, we need to track some terms of the functions $F_{12},H$. Indeed, instead of~\eqref{constants2} we can use bounds of the derivatives of the form
\[
\begin{split}
    \big |\partial_x F_{1,2}(x,y,z) \big | & \leq c_{F}^\phi |z|^2 + K_{F}^\phi |(x,y)||z| \\
    \big |\partial_y F_{1,2}(x,y,z) \big | & \leq c_{F}^\varphi |z|^2 + K_{F}^\varphi |(x,y)||z| \\
    \big |\partial_x H(x,y,z) \big | & \leq c_{H}^\phi |z|^2 + K_{H}^\phi |(x,y)||z|
    \\
    \big |\partial_y H(x,y,z) \big | & \leq c_{H}^\varphi |z|^2 + K_{H}^\varphi |(x,y)||z|
\end{split}
\]
which, together with \eqref{def:estiamatesFH}, imply
\[
\begin{split}
\big |F_{1,2}(x,y,z) +a_3z^3 \big| &\leq C_F^0 |z|^4 + \bar{C}_F |(x,y)||z|^2 + \bar{K}_F|(x,y)|^2 |z| \\
\big |H(x,y,z) +\h z^3 \big| &\leq C_H^0 |z|^4 + \bar{C}_H |(x,y)||z|^2 + \bar{K}_H|(x,y)|^2 |z|
\end{split}
\]
with $\bar{C}_F=c_F^\phi + c_F^{\varphi}$, $\bar{C}_H=c_H^\phi + c_H^{\varphi}$, $\bar{K}_F=K_F^\phi + K_F^{\varphi}$, $\bar{K}_H=K_H^\phi + K_H^{\varphi}$. If necessary, since $|(x,y)|\leq |z|$, we can also use
\[
\begin{split}
\big |F_{1,2}(x,y,z) +a_3z^3 \big| &\leq C_F^0 |z|^4 + \tilde{C}_F |(x,y)||z|^2  \\
\big |H(x,y,z) +\h z^3 + a_4 z^4 + a_5 z^5 \big| &\leq \bar{c}_H^0 |z|^6 + \tilde{C}_H |(x,y)||z|^2.
\end{split}
\]
where
\[
\tilde{C}_F=\bar{C}_F+ \bar{K}_F,\qquad \tilde{C}_H=\bar{C}_H+ \bar{K}_H. \]
It is clear that, taking
\[
\begin{split}
C_F & = |a_3| + \frac{C_F^0}{\rr} + \frac{\tilde{C}_F M_0}{\rr^2}, \qquad
C_F^{\phi,\varphi} = c_{F}^{\phi,\varphi} + \frac{K_F^{\phi,\varphi}M_0}{\rr^2} \\
{C}_H&= |a_4| + \frac{|a_5|}{\rr} + \frac{\bar{c}_H^0}{\rr^2} + \frac{\tilde{C}_H M_0}{\rr},
\qquad
C_H^{\phi,\varphi} = c_{H}^{\phi,\varphi} + \frac{K_H^{\phi,\varphi}M_0}{\rr^2},
\end{split}
\]
we can get a more accurate bound for $|\partial_{\phi,\varphi} \mathcal{S}|$.
In fact, we can just change the definition of $M_{ij}^k$ by changing
the value of $C_F, C_F^{\phi,\varphi}, C_H^{\phi,\varphi}$ by their new value.
\end{remark}

\subsubsection{The Lipschitz constant of $\mathcal{F}^-$}\label{sec:step4}
Now we are going to compute the Lipschitz constant of the operator $\FF^-$ in \eqref{def:F}.
\begin{lemma} Take $\rr \geq \rr_*^2$ as in Lemma~\ref{lem:DSgen}.
The operator $\FF:B(M_0) \to \XX_3\times \XX_3$ is Lipschitz with Lipschitz constant $L(\rr)=\min\{L_1(\rr),L_2(\rr)\}$ with
\begin{equation}\label{def:Ls}
\begin{split}
L_1(\rr)=\,&C_4\frac{M_{11}^1(\rho)}{\rr}+B_{6}
\frac{M_{11}^2(\rr)+ M_{12}^2 (\rr)}{\rr} +
B_{7}
\frac{M_{11}^3(\rr)+ M_{12}^3 (\rr)}{\rr^2}
\\ &+B_{8}
\frac{M_{11}^4(\rr)+ M_{12}^4 (\rr)}{\rr^3} \\
L_2(\rr)=\,&C_4 \frac{M_{11}^1(\rho)}{\rr}+ C_5
\frac{M_{11}^2(\rr)+ M_{12}^2 (\rr)}{\rr^2}  +
C_6
\frac{M_{11}^3(\rr)+ M_{12}^3 (\rr)}{\rr^3}
\\ &+C_7
\frac{M_{11}^4(\rr)+ M_{12}^4 (\rr)}{\rr^4}.
\end{split}
\end{equation}
where $B_\nu$ and $C_\nu$ are the constants introduced in \eqref{def:Bnu} and  \eqref{def:Cnu} respectively.
\end{lemma}
\begin{proof}
We apply the second item of Lemma~\ref{cor:linearB} to
$\Delta \SSS_1(\psi,s)-\Delta \SSS_1(\psi',s)$ and we obtain that
\begin{equation}\label{bound:1Lipschitz}
\|\B^-(\Delta \SSS_1(\psi,s)-\Delta \SSS_1(\psi',s))\|_4\leq C_4M_{11}^1(\rho) \|\phi - \phi'\|_3.
\end{equation}
Now we apply the first item of Lemma~\ref{cor:linearB} to
$\Delta\SSS_j(\psi,s)-\Delta \SSS_j(\psi',s)$ and we obtain
\begin{equation}\label{bound:2Lipschitz}
\|\B^-\left(\Delta\SSS_j(\psi,s)-\Delta \SSS_j(\psi',s)\right)\|_{2+j}\leq B_{j+4}(M_{11}^j(\rr) + M_{12}^j(\rr))\|\phi - \phi'\|_3.
\end{equation}
Then, we get  $L_1(\rr)$ adding the results in~\eqref{bound:1Lipschitz} and~\eqref{bound:2Lipschitz}. Furthermore, applying the second item of Corollary~\ref{cor:linearB}, we obtain $L_2(\rr)$ using that
\begin{equation*}
\|\B^-\left(\Delta\SSS_j(\psi,s)-\Delta \SSS_j(\psi',s)\right)\|_{3+j}\leq C_{j+2}(M_{11}^j(\rr) + M_{12}^j(\rr))\|\phi - \phi'\|_3,
\end{equation*}
\end{proof}

\begin{remark}\label{rmk:L1L2}
%\textbf{Marcel: No entenc aquesta remark per al cas general (hi ha molts termes!!). Per a l'exemple sí.}
Notice that $B_{m}$ is decreasing with respect to $m$, but $C_{m}$ is increasing. It is not difficult to check that when $\rr \geq C_7/B_8\geq 3^9 \cdot 32 /(2^{12}\cdot 5\pi) \sim 9.7895$ then
$L_{1}(\rr)\geq L_{2}(\rr)$. This fact will be used in Section \ref{sec:example1} and \ref{sec:example2}.
\end{remark}

\subsubsection{Setting $\rr_*$ for the existence result}\label{sec:step5}
We choose $\rr_*\geq\rr_*^2$ satisfying
\[
L(\rr_*)\leq \frac{1}{2},
\]
so that Lemma~\ref{lem:DSgen} can be applied for $\rr\geq \rr_*$.
Then, the operator $\FF^- : B(M_0) \to B(M_0)$ is contractive. Indeed,
\[
\|\FF^-(\psi)\|_{3} \leq \|\FF^-(0)\|_{3} + \|\FF^-(\psi)-\FF^-(0)\|_{3} = \frac{M_0}{2} + LM_0\leq M_0
\]
provided $L\leq \frac{1}{2}$. Therefore, the operator has a fixed point $\psi^-$ defined in $\DDD^{-}_{\rr^*}$ (see \eqref{def:domainInnner}) and therefore satisfies
\begin{equation}\label{def:psiestimate}
|\phi^-(s)|,|\varphi^-(s)|\leq \frac{M_0}{|s|^3}.
\end{equation}
%From now on we  work with $\rr \geq \rr%_*^3$

\subsubsection{Explicit bounds for the norm of the linear operator $\GG$}\label{sec:step6}
%\textcolor{red}{Step 6 and 7 are not needed to prove that the $\Theta\neq 0$. It is enough to evaluate $\phi^+-\phi^-$ at $\rho_*^3$.}
The next lemma gives estimates for the  linear operator $\GG$  defined in~\eqref{def:OperatorG} with respect to the norm introduced in~\eqref{def:nomexp}.
%We want to give estimates for $\|\GG\|$.
\begin{lemma}
Take $\rr\geq \rr_*$ and let
\begin{equation}\label{def:A1A2}
\begin{split}
A_1(\rr)&=\frac{M_{11}^2}{\rr} + \frac{M_{11}^3 + M_{12}^2}{2\rr^2} +
\frac{M_{11}^4+M_{12}^3}{3\rr^3}+ \frac{M_{12}^4}{4\rr^4}
\\
A_2(\rr)&=\frac{M_{12}^2}{2\al \rr^2} + \frac{M_{11}^2 + M_{12}^3}{2\al \rr^3} +
\frac{M_{11}^3+M_{12}^4}{2 \al \rr^4}+ \frac{M_{11}^4}{2\al \rr^5}
\end{split}
\end{equation}
where $M_{ij}=M_{ij}(\rr)$ are the constants introduced in \eqref{def:Mij}. Then, we have that, for $s$ with $\Re s=0$ and $\Im s\leq -\rr$,
\begin{equation}\label{def:estimateG}
\begin{split}
\left|s^{-1}e^{i\al (s+\h \log s)}\GG_1(\Delta \psi)\right|\leq&  A_1(\rr)\| \Delta\psi\|\\
\left|e^{i\al (s+\h \log s)}\GG_2(\Delta \psi)\right|\leq &  A_2(\rr)\| \Delta\psi\|.
\end{split}
\end{equation}
In particular,
\begin{equation}\label{eq:boundoperator}
 \|\GG(\Delta \psi)\|\leq  A(\rr)\| \Delta\psi\|,
\end{equation}
with
\begin{equation}\label{def:A}
A=A(\rr) = \max \left \{ A_1(\rr),A_2(\rr)\right \}.
\end{equation}
\end{lemma}
\begin{proof} In this proof we  omit the dependence on $\rho$ of $M_{i,j}^k$. 
We use Lemma~\ref{lem:DSgen} to bound $\widetilde{\KK}_{ij}$, the components of the matrix $\widetilde{\KK}$ in \eqref{def:Ktilde}. By construction, if
$\psi\in B(M_0)$,
\[
\begin{split}
\left|   \widetilde{\KK}_{11}(\psi,s)\right| ,\left|\widetilde{\KK}_{22}(\psi,s)\right| & \leq
   \frac{M_{11}^2}{|s|^2} +  \frac{M_{11}^3}{|s|^3}+  \frac{M_{11}^4}{|s|^4} \\
  \left| \widetilde{\KK}_{12}(\psi,s)\right| ,\left|\widetilde{\KK}_{21}(\psi,s)\right| & \leq
   \frac{M_{12}^2}{|s|^2} +  \frac{M_{12}^3}{|s|^3}+  \frac{M_{12}^4}{|s|^4}.
\end{split}
\]
Then, for the first component,
\begin{equation}\label{def:BoundG1gen}
\begin{aligned}
 \left|s^{-1}e^{i\al (s+\h \log s)}\GG_1(\Delta \psi)\right|&\leq \left|\int_{-i\rr}^s\frac{e^{i\al( t+\h \log t)}}{t}\left(\widetilde{\KK}_{11}\Delta\phi
 + \widetilde{\KK}_{12}\Delta\varphi\right)dt\right|\\
\leq &\int_{-i\rr}^s\left(
\frac{M_{11}^2}{|t|^2}+\frac{M_{11}^3}{|t|^3}+\frac{M_{11}^4}{|t|^4} \right )
\|\Delta\phi\| \, dt \\ &+
\int_{-i \rr}^s \left (
\frac{M_{12}^2}{|t|^3}+\frac{M_{12}^3}{|t|^4}+\frac{M_{12}^4}{|t|^5}\right)
\|\Delta\varphi\|dt\\
\leq &
\left[
\frac{M_{11}^2}{\rr}+\frac{M_{11}^3}{2\rr^2}+\frac{M_{11}^4}{3\rr^3}
\right ]\|\Delta\phi\| +
\left[
\frac{M_{12}^2}{2\rr ^2}+\frac{M_{12}^3}{3 \rr^3}+\frac{M_{12}^4}{4\rr^4}
\right ]\|\Delta\varphi
\\
 \leq &\left (\frac{M_{11}^2}{\rr} + \frac{M_{11}^3 + M_{12}^2}{2\rr^2} +
\frac{M_{11}^4+M_{12}^3}{3\rr^3}+ \frac{M_{12}^4}{4\rr^4}
\right ) \|\Delta \psi\|.
\end{aligned}
\end{equation}
For the second component, using that $\big |e^{i\al \h \log t}\big | = e^{\al \h \pi/2}$,
\[
\begin{split}
 \left|e^{i\al (s+\h \log s)}\GG_2(\Delta \psi)\right|\leq &\left|se^{2i\al (s+\h \log s)}\int_{-\infty}^s\frac{e^{-i\al ( t+\h \log t)}}{t}
 \left(\widetilde{\KK}_{21}\Delta\phi+
 \widetilde{\KK}_{22}\Delta\varphi\right)dt\right|\\
\leq &\left|s\right|e^{-2\al\Im s}\int_{-\infty}^se^{2\al \Im t}
\left(\left[
\frac{M_{12}^2}{|t|^2}+\frac{M_{12}^3}{|t|^3}+\frac{M_{12}^4}{|t|^4}
\right ]\|\Delta\phi\| \right )\, dt \\ & +
|s| e^{-2 \al \Im s} \int_{-\infty}^s e^{2 \al \Im t} \left(
\frac{M_{11}^2}{|t|^3}+\frac{M_{11}^3}{|t|^4}+\frac{M_{11}^4}{|t|^5}
\|\Delta\varphi\|\right)\,dt\\
&\leq
\left[
\frac{M_{12}^2}{2\al \rr^2}+\frac{M_{12}^3}{2\al\rr^3}+\frac{M_{12}^4}{2\al \rr^4}
\right ]\|\Delta\phi\|+
\left[
\frac{M_{11}^2}{2 \al \rr^3}+\frac{M_{11}^3}{2\al  \rr^4}+\frac{M_{11}^4}{2\al\rr^5} \right ]
\|\Delta\varphi \|
\\
& \leq \left (\frac{M_{12}^2}{2\al \rr^2} + \frac{M_{11}^2 + M_{12}^3}{2\al \rr^3} +
\frac{M_{11}^3+M_{12}^4}{2 \al \rr^4}+ \frac{M_{11}^4}{2\al \rr^5}
\right ) \|\Delta \psi\|.
\end{split}
\]
and the result is proven.
\end{proof}

\subsubsection{Computation of the Stokes constant}\label{sec:step7}
Using the estimates of the operator $\GG$ given in \eqref{eq:boundoperator} we can provide a rigorous computation of the Stokes constant.

Let $\rho_0\geq \rho_*$ be such that 
$$
A(\rho_0)< \frac{1}{2}.
$$
Then, the constant $\overline{M}(\rho_0)$ defined in~\eqref{def:Mbarrho} satisfies
\begin{equation}\label{def:Mbarrhosec7}
\overline{M}(\rho_0) = \frac{A_1(\rho_0)}{1-A(\rho_0)} <\frac{A(\rho_0)}{1-A(\rho_0)}<1
\end{equation}
and, as a consequence,  
\[
\Theta\in \left[\kk_0(\rr_0)(1-\overline{M}(\rr_0)), \kk_0(\rr_0)(1+\overline{M}(\rr_0))\right].
\]

In the next section we give the precise algorithm which allows, by means of computer rigorous computations, to compute $\Theta$ with a previous known accuracy.
This algorithm is applied to two  concrete examples in Sections~\ref{sec:example1} and~ \ref{sec:example2}.
%Indeed, by Theorem \ref{th:inner}, it is enough to check that $\psi^+-\psi^-\neq 0$ at $z=\rho^*$. 

\subsection{Computing the Stokes constant: algorithm }\label{sec:AlgorithmMethod1}
We describe the steps needed to obtain the values of $\rr_*$ and $\rho_0$ which guarantees that $\psi^+,\psi^-$ are defined for $s\in (-i\rho_*, -i\infty)$ and a good accuracy of $\Theta$.
\begin{itemize}
    \item Step 1: Compute the constants $a_3,\h$ and $C_{F}^0,C_{H}^0,\overline{C}_H^0$, $C_F^{\phi,\varphi}$, $C_{H}^{\phi,\varphi}$  which satisfy \eqref{def:estiamatesFH} and \eqref{constants2} and $\overline{C}_H$, $\overline{C}_F$ given in \eqref{def:Cbar}.
    \item Step 2: Take  $\rr_*^1 \geq\max\{2, \overline{C}_H^0\}$ and compute, for  $\rr\geq\rr_*^1$, the constants $\mathcal{C}_0(\rr)$ given in Lemma \ref{lem:firstiterationgen} and  $M_0(\rho)$ defined in \eqref{def:M0}.
    \item Step 3: Choose $\rr_*^2\geq \rr_*^1$ such that \eqref{def:minM0} is satisfied.  Compute also  the constants $M^j_{11}(\rho)$, $j=1,2,3,4$ and $M_{12}^j(\rho)$, $j=2,3,4$, in \eqref{def:Mij}, for  $\rr\geq\rr_*^2$.
    \item Step 4: Compute the constants $L_1(\rho)$ and $L_2(\rho)$ in \eqref{def:Ls} for  $\rr\geq\rr_*^2$.
    \item Step 5: Choose $\rr_*\geq \rr_*^2$ satisfying
    \[
    L(\rr_*)=\min\{L_1(\rho_*),L_2(\rho_*)\}\leq \frac{1}{2}.
    \]
    \item Step 6: Take  $\rr_*$ and check that the difference 
    \[
     \psi^+(-i\rr_*)-\psi^-(-i\rr_*)\neq 0.
     \]
     \item Step 7: For $\rr\ge\rr_*$, compute the constants $A_1(\rr)$ and $A_2(\rr)$ in \eqref{def:A1A2} and $A(\rr)$ in \eqref{def:A}.
     \item Step 8: Compute $\rr_0\ge\rr_*$ such that $A(\rr_0)\leq 1/2$. Then, %, for $\rr>\rr_0$,
     compute $\kk_0(\rr_0)$ in \eqref{def:kapparho} and $\overline{M}(\rr_0)$ in~\eqref{def:Mbarrhosec7}. 
         \end{itemize} 
Therefore, the Stokes constant satisfies
     \begin{equation}
\Theta\in \left[\kk_0(\rr_0)(1-\overline{M}(\rr_0)), \kk_0(\rr_0)(1+\overline{M}(\rr_0))\right]. \label{eq:Stokes-step-8}
\end{equation}
\begin{remark}
By Theorem~\ref{th:inner}, the first 6 steps allows us to check whether $\Theta\neq 0$ or not.
\end{remark}

\section{Examples}\label{sec:examples}
To illustrate the algorithm we consider two concrete examples of analytic unfoldings of a Hopf-Zero singularity~\eqref{initsys} whose corresponding inner equation can be found in~\eqref{almostinner}.
In both cases, we prove that the associated constants $\Theta$ does not vanish and give rigorous estimates for them. 

\subsection{The first example}\label{sec:example1}
 As first example we take
\begin{equation}\label{def:Parameters}
\al=1,\qquad \qquad b=1,\qquad g=h=0\qquad \text{and}\qquad f(X,Y,Z,\d)=Z^3.
\end{equation}
This  corresponds to  $F_1(\phi,\varphi,s)=-s^{-3}$, $F_2=F_1$ and $H=0$.
The inner equation \eqref{eq:InnerModifiedgen} associated to this model is the following
\begin{equation}\label{eq:InnerModified1}
 \begin{pmatrix} \frac{d\phi}{ds}\\ \frac{d\varphi}{ds}\end{pmatrix}=\AAA(s) \begin{pmatrix}\phi\\ \varphi\end{pmatrix}+\SSS(\phi,\varphi,s)
\end{equation}
with
\begin{equation}\label{def:R}
 \SSS(\phi,\varphi,s)=\begin{pmatrix}
                  \SSS_1(\phi,\varphi,s)\\
                  \SSS_2(\phi,\varphi,s)
                   \end{pmatrix}
=
 \begin{pmatrix}
                \dps   \frac{\left(i-\frac{1}{s}\right)\varphi\phi^2
s^2-\frac{1}{s^3}}{1+\varphi\phi s^2}\\
                \dps   \frac{\left(i+\frac{1}{s}\right)\varphi^2\phi
s^2-\frac{1}{s^3}}{1+\varphi\phi s^2}
                   \end{pmatrix}
\end{equation}

%Note that the chosen unfolding is reversible with respect to the involution
%\begin{equation}\label{def:involution}
 %\mathcal{I}(\phi,\varphi,s)=(\varphi,\phi, s).
%\end{equation}
%This is not crucial but makes the analysis simpler.

Now we follow the steps of the algorithm in Section~\ref{sec:AlgorithmMethod1}. 

\noindent
\textbf{Step 1.}
In this case $\h=0$ and moreover, among all the constants defined in Step 1, the only one that is different form $0$ is $a_3=1$.

\noindent
\textbf{Step 2.} In this case we have that $\rr_*^1=2$ and $\mathcal{C}_0=0$ so that $M_0=\frac{22}{3}$ is independent on $\rho$.

\noindent
\textbf{Step 3.} We have that $\rr_*^2$ has to be such that
$$
\sqrt{M_0}=\sqrt{\frac{22}{3}} < \rho_*^2.
$$
In addition $M_{ij}^{1, 2, 3}=0$ and
\[
\begin{split}
M_{11}^4(\rho) &= \frac{M_0}{\left (1 - \frac{M_0^2}{\rr^4}\right )^2}  \left (
1 + M_0 + \left (M_0 + \frac{M_0}{\rr}\right )\left (1 - \frac{M_0^2}{\rr^4}\right ) + \frac{M_0}{\rr}\right )
\\&= \frac{M_0} {\left (1 - \frac{M_0^2}{\rr^4}\right )^2}
\left (1 + M_0\left (1+ \frac{1}{\rr}\right ) \left ( 2- \frac{M_0^2}{\rr^4}\right ) \right )
\\
M_{12}^4(\rho) &= \frac{M_0}{\left (1 - \frac{M_0^2}{\rr^4}\right )^2}  \left (
1 + M_0 \left (1 + \frac{1}{\rr}\right ) \right ).
\end{split}
\]

\noindent
\textbf{Step 4 and Step 5.}
One can check that
$$
L_1(\rr)=B_8\frac{M_{11}^4(\rr) + M_{12}^4(\rr)}{\rr^3}\leq \frac{1}{2}
$$
for $\rr \geq \rr_*=9.7895$.
%we obtain that $\rr\geq \rr_* \geq 5.9850$. 
Under this condition, as we claimed in Remark~\ref{rmk:L1L2},
\[
L(\rr)\leq \frac{1}{2}.
\]
Therefore we can guarantee the existence of $\psi^{\pm}$ for $\rr \geq \rr_*=9.7895$.

%\textbf{No s'hauria de calcular $L_2$ o agafar $\rho\geq 10$?}
%\noindent

\noindent\textbf{Step 6.} Now it only remains to compute 
 \[
     \psi^+(-i\rr_*)-\psi^-(-i\rr_*)\neq 0.
     \]
By  means of rigorous computer computations, which are discussed in more detail in Appendix \ref{app:Delta-comp}, we obtain that there exists a 
\begin{equation}\rho_* \in  [15.99999965, 16.00000035] \label{eq:rho-ex1}
\end{equation}
for which
\begin{align}
& \psi^{+}\left(  -i\rho_{\ast
}\right)  -\psi^{-}\left(  -i\rho_{\ast}\right) \label{eq:Delta-ex1} \\
& =\left(
\begin{array}
[c]{c}%
\Delta\phi\left(  -i\rho_{\ast}\right)  \\
\Delta\varphi\left(  -i\rho_{\ast}\right)
\end{array}
\right) \notag \\
& \in \left(
\begin{array}
[c]{c}
\lbrack-4.50096\cdot10^{-10}, 4.50096\cdot10^{-10}]-[1.88812\cdot10^{-6},1.88897\cdot10^{-6}]i\\
\lbrack-3.85539\cdot10^{-10}, 3.85539\cdot10^{-10}]+[-4.01544\cdot10^{-10}, 3.4832\cdot10^{-10}]i
\end{array}
\right). \notag
\end{align}

%\textbf{Write the results}
Therefore, the Stokes constant associated to the first example \eqref{def:Parameters} does not vanish. 

Now we follow \textbf{Step 7.} and \textbf{Step 8.} to provide rigorous accurate estimates for it.

\noindent
\textbf{Step 7.}
The constants $A_1$ and $A_2$ in \eqref{def:A1A2} are
\begin{equation}\label{def:A:Example1}
A_1(\rr)=
\frac{M_{11}^4(\rho)}{3\rr^3}+ \frac{M_{12}^4(\rho)}{4\rr^4},
\qquad 
A_2(\rr)=\frac{M_{12}^4(\rho)}{2 \rr^4}+ \frac{M_{11}^4(\rho)}{2\rr^5}
\end{equation}
which give  the constant $A(\rr)=\max\{A_1(\rr),A_2(\rr)\}$ in \eqref{def:A}. We obtain 
\begin{align*}
A_1(\rho_*) & \in [0.010155523, 0.010155525], \\
A_2(\rho_*) & \in [0.0009597786, 0.0009597788], \\
A(\rho_*)  & = A_1(\rho_*) < 1/2.
\end{align*}

\noindent
\textbf{Step 8.} One can choose $\rho_0=\rr_*$. Then, by (\ref{def:kapparho}), (\ref{def:Mbarrhosec7}) and (\ref{eq:Stokes-step-8}), one obtains
\begin{equation}
\Theta\in [1.0378681, 1.0598665]+[-0.000253, 0.000253]i\label{eq:S-initial}
\end{equation}
We can see that the accuracy of the computation is roughly $2\cdot 10^{-2}$.

\subsection{The second example}\label{sec:example2}
The second example, breaks the reversibility. It consists in taking $\alpha=b=1$, $g=h=0$ and $f(X,Y,Z,\d)=Z^3+ 2XYZ$ which corresponds to
$$
F_1(\phi,\varphi,s) = F_2(\phi,\varphi,s)= -\frac{1}{s^3}+ \frac{i}{s} (\phi^2 -\varphi^2),\qquad H=0
$$
and then, the inner equation associated to this unfolding is:
\begin{equation}\label{eq:InnerModified2}
 \begin{pmatrix} \frac{d\phi}{ds}\\ \frac{d\varphi}{ds}\end{pmatrix}=\mathcal{A}(s) \begin{pmatrix}\phi\\ \varphi\end{pmatrix}+\mathcal{S}(\phi,\varphi,s)
\end{equation}
with $\mathcal{S}$ defined as 
\begin{equation*}
 \SSS(\phi,\varphi,s)=\begin{pmatrix}
                  \SSS_1(\phi,\varphi,s)\\
                  \SSS_2(\phi,\varphi,s)
                   \end{pmatrix}
=
 \begin{pmatrix}
                \dps   \frac{\left(i-\frac{1}{s}\right)\varphi\phi^2
s^2-\frac{1}{s^3}+\frac{i}{2s} (\phi^2 - \varphi^2)}{1+\varphi\phi s^2}\\
                \dps   \frac{\left(i+\frac{1}{s}\right)\varphi^2\phi
s^2-\frac{1}{s^3}+ \frac{i}{2s}(\phi^2 - \varphi^2)}{1+\varphi\phi s^2}
                   \end{pmatrix}.
\end{equation*}

\noindent
\textbf{Step 1.}
We have that all the constants are zero except $a_3=1$, $C_F=2$, $C_F^\phi=C_F^\varphi=1$.

\noindent
\textbf{Step 2.} As for the first example $\rr_*^1=2$ and $\mathcal{C}_0=0$ so that $M_0=\frac{22}{3}$.

\noindent
\textbf{Step 3.} We have that $\rr_*^2$ has to be such that
$$
\sqrt{M_0}=\sqrt{\frac{22}{3}} < \rho_*^2.
$$
In addition $M_{ij}^{1, 3}(\rho)=0$ being
\[
\begin{split}
M_{11}^2(\rho) &= M_{12}^2 (\rho)= \frac{2}{1- \frac{M_0^2}{\rr^4}}\\
M_{11}^4(\rho) &= \frac{M_0}{\left (1 - \frac{M_0^2}{\rr^4}\right )^2}  \left (
2 + M_0 + \left (M_0 + \frac{M_0}{\rr}\right )\left (1 - \frac{M_0^2}{\rr^4}\right ) + \frac{M_0}{\rr}\right )
\\&= \frac{M_0} {\left (1 - \frac{M_0^2}{\rr^4}\right )^2}
\left (2 + M_0\left (1+ \frac{1}{\rr}\right ) \left ( 2- \frac{M_0^2}{\rr^4}\right ) \right )
\\
M_{12}^4(\rho) &= \frac{M_0}{\left (1 - \frac{M_0^2}{\rr^4}\right )^2}  \left (
2+ M_0 \left (1 + \frac{1}{\rr}\right ) \right ).
\end{split}
\]

\noindent
\textbf{Step 4 and Step 5.}
One can check that 
$$
L_1(\rr)=B_6 \frac{M_{11}^2(\rho)+ M_{12}^2(\rho)}{\rr} + B_8\frac{M_{11}^4(\rr) + M_{12}^4(\rr)}{\rr^3}\leq \frac{1}{2}
$$
for $\rr\geq \rr_* \geq 9.7895$. Therefore, under this condition, using Remark \ref{rmk:L1L2} as for Example 1,  we can guarantee that $L(\rr)\geq 1/2$ and, then, the existence of $\psi^{\pm}$.

\begin{remark} \label{rmk:secondexample}
For  Example 2, we  can obtain more accurate estimates by computing directly the derivatives $\partial_{\phi,\varphi} \mathcal{S}$. Indeed,
performing straightforward computations
we obtain that $M_{ij}^1(\rr)=M_{ij}^2 (\rr) = M_{ij}^3(\rr)=0$ and
\begin{equation}\label{betterboundsE2}
\begin{split}
    M_{11}^4(\rr) &= \frac{M_0} {\left (1 - \frac{M_0^2}{\rr^4}\right )^2}
\left (3 + M_0 \left (1+ \frac{1}{\rr}\right ) \left ( 2- \frac{M_0^2}{\rr^4}\right )\right )  \\
    M_{12}^4(\rr) &= \frac{M_0}{\left (1 - \frac{M_0^2}{\rr^4}\right )^2}  \left (
3 + M_0 \left (1 + \frac{1}{\rr}\right ) \right ).
\end{split}
\end{equation}

Therefore
\[
L_1(\rr)= B_8 \frac{M_{11}^4(\rr) + M_{12}^4(\rr)}{\rr^3}=
\frac{5\pi}{32} \frac{M_{11}^4(\rr) + M_{12}^4(\rr)}{\rr^3}.
\]
\end{remark}

Since we need to be as precise as possible, we  use the constants $M_{ij}^k(\rho)$ defined in Remark~\ref{rmk:secondexample} instead of the constants provided by the general method. 

\noindent\textbf{Step 6.} We compute
$\psi^+(-i\rr_*)-\psi^-(-i\rr_*)$. By  means of rigorous computer computations we obtain that there exists a $\rho_*$
\begin{equation}
\rho_*\in [15.99999965, 16.00000035], \label{eq:rho-ex2}
\end{equation}
for which
\begin{align}
& \psi^{+}\left(  -i\rho_{\ast
}\right)  -\psi^{-}\left(  -i\rho_{\ast}\right)  \label{eq:Delta-ex2}\\
&  =\left(
\begin{array}
[c]{c}%
\Delta\phi\left(  -i\rho_{\ast}\right)  \\
\Delta\varphi\left(  -i\rho_{\ast}\right)
\end{array}
\right)  \nonumber\\
&  \left(
\begin{array}
[c]{c}%
\lbrack8.63066\cdot10^{-9}, 9.53086\cdot10^{-9}]-[ 1.88812\cdot10^{-6},1.88897\cdot10^{-6}]i\\
\lbrack-4.20777\cdot10^{-10}, 3.50313\cdot10^{-10}]+[-4.01721\cdot10^{-10}, 3.48156\cdot10^{-10}]i
\end{array}
\right).  \nonumber
\end{align}

This implies that $\Theta\neq 0$. Now we perform the last two steps in the algorithm.

\noindent
\textbf{Step 7.} For Example 2, we have that
\[
A_1(\rr)=\frac{M_{11}^4(\rr)}{3\rr^3}+ \frac{M_{12}^4(\rr)}{4\rr^4},
\qquad 
A_2(\rr)=\frac{M_{12}^4(\rr)}{2  \rr^4}+ \frac{M_{11}^4(\rr)}{2 \rr^5},
\]
and $A(\rr)=\max \left \{ A_1(\rr),A_2(\rr)\right \}$, 
with $M_{11}^4(\rr),M_{12}^4(\rr)$ defined in~\eqref{betterboundsE2}. We obtain
\begin{align*}
A_1(\rho_*) &\in [0.0114071016, 0.0114071033], \\
A_2(\rho_*) &\in [0.00066984056, 0.00066984069], \\
A(\rho_*)&=A_1(\rho_*) < 1/2.
\end{align*}

\noindent
\textbf{Step 8.} 
By means of rigorous computer validation, for $\rr_0=\rr_*$ using (\ref{def:kapparho}), (\ref{def:Mbarrhosec7}) and (\ref{eq:Stokes-step-8}) we obtain
\begin{equation*}
\Theta\in [1.036525, 1.0612062]+[0.004738, 0.005355]i.
\end{equation*}
We can see that the accuracy of the computation is roughly $2.5\cdot 10^{-1}$.

\section{Improving the computation of the Stokes constant}\label{sec:StokesComputeExtrap}
%\section{A method to compute the Stokes constant}

In this section we give an improvement of the  Steps 7 and 8 in Section~\ref{sec:AlgorithmMethod1} to obtain accurate estimates for the Stokes constant $\Theta$. We explain this improvement for the Example 1 given in Section \ref{sec:example1} but the method we present is general and can be applied to any system.

Recall that, using \eqref{def:StokesLimit} and  \eqref{eq:diferenciapuntfix},
\[
\begin{split}
 \Theta &=\lim_{\Im s\to-\infty}s^{-1}e^{i\al s}\Delta\phi(s)=\kk_0+\lim_{\Im s\to-\infty}s^{-1}e^{i\al s}\GG_1\left(\Delta\psi(s)\right)\\
 &=
 \kk_0+\lim_{\Im s\to-\infty}s^{-1}e^{i\al s}\GG_1\left(\Delta\psi_0(s)\right)+
 \lim_{\Im s\to-\infty}s^{-1}e^{i\al s}\GG_1\left(\GG(\Delta\psi)(s)\right).
\end{split}
\]
Therefore, by \eqref{def:estimateG}, \eqref{eq:boundoperator}  and \eqref{def:Estkappa0}, the remainder
\[
\mathcal{E}_{\Theta}=\Theta-\kk_0-\lim_{\Im s\to-\infty}s^{-1}e^{i\al s}\GG_1\left(\Delta\psi_0(s)\right)
\]
satisfies
\[
\begin{split}
\left|\mathcal{E}_{\Theta}\right|&\le
 \sup_s\left|s^{-1}e^{i\al s}\GG_1\left(\GG(\Delta\psi)(s)\right)\right|  \le  A_1 \|\GG(\Delta \psi)\|
 \le A_1 A \|\Delta\psi\|\\
 &\leq  A_1A  \frac{|\kk_0|}{1-A}.
\end{split}
\]
where $A=A(\rho)=\max\{A_1,A_2\}$ and  $A_1, A_2$  are given in~\eqref{def:A:Example1}.

Using \eqref{def:OperatorG} and the fact that in Example 1, $h_0=0$ and therefore $\widetilde{\mathcal{K}}=\mathcal{K}$,
\begin{equation}\label{eq:improvtheta}
\begin{split}
 \lim_{\Im s\to-\infty}s^{-1}e^{i\al s}\GG_1\left(\Delta\psi_0(s)\right)
 =-\kk_0\int_{-i\infty}^{i\rho}\KK_{11}(t)dt = -i\kk_0\int_{-\infty}^{\rho}\KK_{11}(i\, r)dr,
\end{split}
\end{equation}
which implies
\begin{equation}\label{ThetaImprovement}
%\b%egin{split}
 \Theta  =
 \kk_0 -i\kk_0\int_{-\infty}^{\rho}\KK_{11}(i\, r)dr +\mathcal{E}_{\Theta}.
% \\
% \mathcal{E}_{\Theta}| & \le A_1A  \frac{\kk_0}{1-A}
 % \|\Delta\psi\|
%\%end{split}
\end{equation}
To obtain this integral we need an  approximation of  the coefficient $\KK_{11}$.

\begin{lemma}\label{lemma:K11Computation}
The function $\KK_{11}$  introduced in \eqref{def:K} associated to equation \eqref{eq:InnerModified1} satisfies 
\begin{equation}\label{formulaK11:0}
\begin{split}
\KK_{11}(s) &=
\frac{i\beta_1 }{s^4} +\frac{\beta_2}{ s^5}+
\frac{i\, \beta_3  }{ s^6}+\frac{\beta_4}{ s^7}
+ \mathcal{EKT}_{11}\\
\beta_1&=3, \ \beta_2= -6, \
\beta _3=-68, \beta_4=48
\end{split}
\end{equation}
and $\mathcal{EKT}_{11}$ satisfies
\begin{equation}\label{errorK11total}
|\mathcal{EKT}_{11}| \le \frac{B\RRR}{|s|^7}+\frac{B_{11}+B_{12}+B_{13}+B_{14}}{|s|^8}
\end{equation}
%\[
%\mathcal{EKT}_{11}= \mathcal{EK}_{11}
%+
%\mathcal{ER}_1 +\mathcal{ER}_2+\mathcal{ER}_3+\mathcal{ER}_4
%\]
where
\[
\begin{split}
B&=\frac{5\pi}{32} \left(M^4_{11}(\rr)+ M_{12}^4(\rr)\right)M_0+
  \frac{225 \pi }{2 }\\
\RRR&=\frac{
1+ 4 M_0(1+\frac{1}{\rho})+
3\frac{ M_0^2}{\rho^4}} {\left(1-\frac{M_0^2}{\rr^4}\right)^3}\\
B_{11}&= \frac{M_*^4 (1+\frac{1}{\rho})}{\left(1-\frac{M_*^2}{\rr^4}\right)^2}\\
  B_{12}&=M_*^5\frac{ (2M_*(1+\frac{1}{\rho})+1)(3+\frac{2M_*^2}{\rho^4})}
{\rho^4\left(1-\frac{M_*^2}{\rr^4}\right)^2}\\
  B_{13}&=2  M_*^3 \left (2M_* \left  (1+\frac{1}{\rho}\right )+1 \right )\\
  B_{14}&=800\left(1-\frac{1}{\rho}\right) 
\end{split}
\]
and
\[
M_*(\rr)=1+\frac{4}{\rr}+\frac{20}{\rr^2}+\frac{120}{\rr^3}.
\]
\end{lemma}

This lemma is proven in  Section \ref{sec:Kcomputation}.

Now take $\rho_{\diamond}$ be such that $M_{*}(\rho)\leq M_0(\rho)$ for $\rho\geq \rho_{\diamond}$.
Using the expression of $\KK_{11}$ given in \eqref{formulaK11:0}  to compute   \eqref{eq:improvtheta} and using the remainder estimates in \eqref{errorK11total} we obtain:
\begin{equation}
\begin{split}
 -i\kk_0\int_{-\infty}^{\rho}\KK_{11}(i\, r)dr
 &= \kk_0\int_{-\infty}^{\rho}\frac{\beta_1}{r^4}-\frac{\beta_2}{r^5}-
\frac{\beta_3  }{r^6}+
\frac{\beta_4 }{r^7}
dr
+\mathcal{ES}\\
&=\kk_0\left(-\frac{\beta_1}{3 \rho^3}+\frac{\beta_2}{4\rho^4}-
\frac{\beta_3  }{5 \rho^5}-
\frac{\beta_4 }{6\rho^6}\right)
+\mathcal{ES}
\end{split}
\end{equation}
and
\[
\begin{split}
|\mathcal{ES}|&=|\mathcal{ES}(\rho)| \le \kk_0 \int_{\rho}^{\infty}
\frac{B\RRR}{r^7}+\frac{B_{11}+B_{12}+B_{13}+B_{14}}{r^8}
\\
&= \frac{B\RRR}{6\rho ^6}+\frac{B_{11}+B_{12}+B_{13}+B_{14}}{7\rho^7}.  
\end{split}
\]
Finally, using the expression~\eqref{ThetaImprovement} for $\Theta $ we obtain
\[
\begin{split}
 \Theta  &=
 \kk_0\left(1-\frac{\beta_1}{3 \rho^3}+\frac{\beta_2}{4\rho^4}-
\frac{\beta_3  }{5 \rho^5}
-\frac{\beta_4 }{6\rho^6}
\right)
+\mathcal{ET},\\
\mathcal{ET}&= \mathcal{ET}(\rho)= \mathcal{ES} +\mathcal{E}_{\Theta}.
 \end{split}
\]
As we know the constants $\beta_i$, one can use this formula to improve the computation of $\Theta$.

Indeed, using the above approach one obtains
\begin{equation}
\Theta\in [1.047906, 1.049289]+[-0.00070294, 0.00070294]i. \label{eq:S-improved}
\end{equation} 
We can see that the accuracy of the computation is roughly $10^{-3}$, which is an improvement when compared to the accuracy $2\cdot 10^{-1}$ from (\ref{eq:S-initial}). (For (\ref{eq:S-improved}) we have used the same $\rho_*$ and the computed value of $\Delta\psi(-i\rho_*)$ as for (\ref{eq:S-initial}).)

%An alternative is to use extrapolation. For instance, consider $\gamma>0$ and write the formula for $\rho $ and $\gamma \rho$:
%\[\begin{split} \Theta  &= %\kk_0\left(1-\frac{\beta_1}{3\rho^3}+\frac{\beta_2}{4\rho^4}-%\%frac{\beta_3  }{5 \rho^5}
%-\frac{\beta_4 }{6\rho^6}
%\right)+\mathcal{ET}(\rho),\\ \Theta  &=
% \kk_0\left(1-\frac{\beta_1}{3 (\gamma %\rho)^3}+\frac{\beta_2}{4(\gamma \rho)^4}-
%\frac{\beta_3  }{5 (\gamma \rho)^5}
%-\frac{\beta_4 }{6(\gamma \rho)^6}\right)+\mathcal{ET}(\gamma %\rho) \end{split}\]
%multiplying the second equation by $\gamma^3$ andsubstracting we get:\[\begin{split}
%(1-\gamma^3) \Theta  &=\kk_0 \left(
%(1-\gamma^3) +\frac{\beta_2}{4\rho^4}(1-\frac{1}{\gamma})-
%\frac{\beta_3  }{5 \rho^5}(1-\frac{1}{\gamma^2})
%-\frac{\beta_4 }{6\rho^6}(1-\frac{1}{\gamma^3})
%\right)\\
%&+\mathcal{ET}(\rho)-\gamma^3 \mathcal{ET}(\gamma \rho)
% \end{split}\]
%Which gives a new formula for $\Theta$:
%\[\begin{split}
%\Theta  &=\frac{\kk_0}{(1-\gamma^3) } \left(
%(1-\gamma^3) +\frac{\beta_2}{4\rho^4}(1-\frac{1}{\gamma})-
%\frac{\beta_3  }{5 \rho^5}(1-\frac{1}{\gamma^2})
%-\frac{\beta_4 }{6\rho^6}(1-\frac{1}{\gamma^3})
%\right)\\
%&+\frac{\mathcal{ET}(\rho)-\gamma^3 \mathcal{ET}(\gamma %\rho)}{(1-\gamma^3) }
% \end{split}\]
%In this way we can eliminate all the explicit terms from %$\Theta$ without computing them.

\subsection{Proof of Lemma~\ref{lemma:K11Computation} }\label{sec:Kcomputation}
Lets call $\KK_{ij}(s)$ the $4$ elements of the matrix $\KK$. To obtain expansions for these coefficients we compute first an expansion for $\psi^\pm$ associated to Example 1 in \eqref{def:Parameters}.

\begin{lemma}\label{lemma:psiexpansion}
The functions $\psi^\pm$  can be written as 
\[
\psi^\pm=\psi_*+\EE^\pm
\]
where $\psi_*=(\phi_*,\varphi_*)$ 
with 
\[
\begin{split}
\phi_* &= -\frac{i}{s^3} -\frac{4}{ s^4}+\frac{20i}{s^5}+\frac{120}{s^6} \\
\varphi_* &= \frac{i}{ s^3} -\frac{4}{s^4}-\frac{20i}{s^5} +\frac{120}{s^6},
\end{split}
\]
which satisfy
\begin{equation}\label{def:psisgtarestimate}
|\phi_*(s)|,|\varphi_*(s)|\leq \frac{M_*(\rr)}{|s|^3} 
\end{equation}
for
\begin{equation}\label{def:Mstar}
M_*(\rr)=1+\frac{4}{\rr}+\frac{20}{\rr^2}+\frac{120}{\rr^3}
\end{equation}
and the remainders $\EE^\pm=(\EE^\pm_\phi,\EE^\pm_\varphi)$ satisfy
\begin{equation}\label{def:Ephiestimate}
|\EE^\pm_\phi|,|\EE^\pm_\varphi|\le
\frac{B}{|s|^6}.
\end{equation}
with
\[
B=\frac{5\pi}{32} \left(M^4_{11}(\rr)+ M_{12}^4(\rr)\right)M_0+
  \frac{225 \pi }{2 }.
\]
\end{lemma}
\begin{proof}
By \eqref{def:R}, the first iteration $\FF^-(0)$ analyzed in Lemma \ref{lem:firstiterationgen} for Example 1 is given by
\[
\begin{split}
\FF_1^-(0)&=s\int_{-\infty}^0\frac{1}{(s+t)^4}
 e^{i t}dt\\
\FF_2^-(0)&=s\int_{-\infty}^0\frac{1}{(s+t)^4}
 e^{-i t}dt
\end{split}
\]
Integrating by parts, we obtain
\[
\begin{split}
 \FF_1^-(0)=&\, -\frac{i}{ s^3}+\frac{4s}{i}\int_{-\infty}^0\frac{1}{(s+t)^5}
 e^{i t}dt\\
=&\, -\frac{i}{ s^3} -\frac{4}{ s^4}+\frac{20i}{ s^5} +\frac{120}{s^6}+ E_1^-(s)
=\phi_*(s)+E_1^-(s),
\\
\end{split}
 \]
 where
\[
|E_1^-(s)|=\left|720 s\int_{-\infty}^0\frac{1}{(s+t)^8}e^{i  t}dt\right| \le \frac{720}{|s|^6}\int_0^\infty \frac{1}{(t^2+1)^4}dr = \frac{720}{|s|^6} \frac{5 \pi }{32}=\frac{225 \pi }{2 |s|^6}.
\]
Analogously, for the second component
\[
\begin{split}
 \FF_2^-(0)=\, \frac{i}{s^3} -\frac{4}{ s^4}+\frac{20i}{s^5}
 +\frac{120}{s^6} +E_2^-(s)
 = \varphi_*(s)+E_2^-(s),
\end{split}
 \]
 where $E_2^-$ has the same bounds as $E_1^-$:
\[
|E_2^-(s)|\le \frac{225 \pi }{2|s|^6}.
\]
Let us call $\psi_*=(\phi_*,\varphi_*)$ and $E^-=\left(E_1^-,E_2^-\right)$. Observe that by Lemma \ref{cor:linearB} and Corollary \ref{coro:Sbounds} and recalling that, for Example 1, one has $M_{11}^j=M_{12}^j=0$ for $j=1,2,3$,
%the bounds in \eqref{eq:fitesFF} (beginning at the second inequality), putting $\psi'=0$ give:
%
\[
 \begin{split}
\left\|\FF^-(\psi)-\FF^-(0)\right\|_{6}
 &\leq\frac{5\pi}{32}\left\|\SSS(\psi,s)-\SSS(0,s)\right\|_{7}\\
 &\leq\frac{5\pi}{32} \left(M^4_{11}(\rr)+ M_{12}^4\right)\left\|\psi\right\|_{3}.
 %\\
 %&\leq\frac{5\pi}{32} \max\{M_{11}+M_{21}, M_{12}+M_{22}\}M_0
 \end{split}
\]
Now we use that $\psi ^-$ is a fixed point of operator $\FF^-$ and therefore
%, by   \eqref{def:operator7}:
\[
 \begin{split}
\left\|\psi^- -\psi_*\right\|_{6}& \le
\left\|\psi^-  -\FF^-(0)\right\|_{6}+ \left\|E^-(0)\right\|_{6}\\
  &\leq\frac{5\pi}{32} \left(M^4_{11}(\rr)+ M_{12}^4(\rr)\right)M_0+
  \frac{225 \pi }{2 }   =B.
\end{split}
\]
We conclude
\[
\begin{split}
\left|\phi ^- -\left(- \frac{i}{s^3} -\frac{4}{s^4}-\frac{20i}{s^5} +\frac{120}{s^6}\right) \right| & \le \frac{B}{|s|^6} \\
\left|\varphi^- -\left( \frac{i}{s^3} -\frac{4}{s^4}+\frac{20i}{ s^5} +\frac{120}{ s^6}\right)\right| & \le \frac{B}{|s|^6}.
\end{split}
 \]
%and, therefore,
%\[
%\begin{split}
%\left|\phi^- -\left( - \frac{i}{s^3} -\frac{4}{s^4}-\frac{20i}{s^5} %\right)\right| & \le
%\frac{\bar B}{|s|^6}\\
%\left|\varphi^- -\left(\frac{i}{s^3} -\frac{4}{s^4}+\frac{20i}{ %\frac{\bar B}{|s|^6}
%\end{split}
% \]
%where $\bar B= B+120$.
\end{proof}

Using the previous result, we can compute a better asymptotic expansion of $\KK_{11}$. We rely on the expression
\[
\KK_{11}(s)=\int_0^1 \partial_\phi\SSS_{1}(\psi^-(s)+t(\psi^+(s) -\psi^-(s)) dt.
\]
Note that, for Example 1, 
\[
\pa_\phi \SSS_1=\frac{2\tA\varphi\phi +\tB \varphi^2\phi^2+ \tC\varphi}{\left(1+\tD\varphi\phi\right)^2}
\]
where 
\[
\tA=\left(i-\frac{1}{s}\right)s^2, \quad \tB=s^4 \left(i-\frac{1}{s}\right), \quad \tC=\frac{1}{s},\quad \tD=s^2.
\]
We define the function
\[
g(r)=\pa_\phi\SSS_{1}\left(\psi_*(s)+r\left( \EE^-(s)+t\left(\EE^+(s) -\EE^-(s)\right)\right)\right)
\]
and using the fundamental theorem of calculus, $g(1)=g(0)+\int_0^1 g'(r)dr$, we have that 
\[
\begin{split}
& \pa_\phi\SSS_1(\psi^-(s)+t(\psi^+(s) -\psi^-(s)) =\pa_\phi \SSS_1(\psi_*(s))\\
&+
\left( \EE_\phi^-(s)+t(\EE_\phi^+(s) -\EE_\phi^-(s)\right)
\int_0^1\partial_{\phi\phi}  \SSS_1 \left(\psi_*(s)+r\left( \EE^-(s)+t(\EE^+(s) -\EE^-(s)\right)\right)
\ dr \\
&+
\left( \EE_\varphi^-(s)+t(\EE_\varphi^+(s) -\EE_\varphi^-(s)\right)
\int_0^1\partial_{\phi\varphi} \SSS_1 \left(\psi_*(s)+r\left( \EE^-(s)+t(\EE^+(s) -\EE^-(s)\right)\right)
\ dr.
\end{split}
\]
Therefore,
\begin{equation}\label{eq:Kijaprox}
\begin{split}
& \KK_{11}=\pa_\phi \SSS_1(\psi_*(s))\\
&+\int_0^1 \left[
\left( \EE_\phi^-(s)+t(\EE_\phi^+(s) -\EE_\phi^-(s)\right)
\int_0^1\partial_{\phi\phi} \SSS_{1}\left(\psi_*(s)+r\left( \EE^-(s)+t(\EE^+(s) -\EE^-(s)\right)\right)\, dr
\right]\, dt\\
&+
\int_0^1\left[ \left( \EE_\varphi^-(s)+t(\EE_\varphi^+(s) -\EE_\varphi^-(s)\right)
\int_0^1\partial_{\phi\varphi} \SSS_{1}\left(\psi_*(s)+r\left( \EE^-(s)+t(\EE^+(s) -\EE^-(s)\right)\right)
\, dr \right]\, dt.
\end{split}
\end{equation}
%To bound these integrals let us introduce the notation
%\[
%\begin{split}
%A&=\left(i\al-\frac{1}{s}\right)bs^2, \ B=b^2s^4 %\left(i\al-\frac{1}{s}\right), \ C=\frac{b}{s}\\
%\tilde A&=\left(i\al+\frac{1}{s}\right)bs^2, \ \tilde B=b^2s^4 \left(i\al+\frac{1}{s}\right), \ D=b s^2= C s^3
%\end{split}
%\]
%With this notation the derivatives of $R=(R_1,R_2)$ read
%\[
%%\begin{split}
%\RRR_{11}(\psi)&= D_\phi R_1=\frac{2A\varphi\phi +B \varphi^2\phi^2+ C\varphi}{\left(1+D\varphi\phi\right)^2}\\
%\RRR_{12}(\psi)&= D_\varphi R_1=\frac{A \phi^2+C\phi}{\left(1+D\varphi\phi\right)^2} \\
%\RRR_{21}(\psi)&= D_\phi R_2=\frac{\tilde A \varphi^2+ C \varphi}{\left(1+D\varphi\phi \right)^2}\\
%\RRR_{22}(\psi)&= D_\varphi R_2=\frac{2\tilde A \varphi\phi+ \tilde B \varphi^2\phi s^2+ C \phi}{\left(1+D\varphi\phi \right)^2}\\
%\end{split}
% \]
One can easily check that
\[
%\begin{split}
\partial_{\phi\phi} \SSS_{1}(\psi)= \frac{2\tA\varphi  - 2\tD\tC\varphi^2}{\left(1+\tD\varphi\phi\right)^3},\qquad 
\partial_{\phi\varphi} \SSS_{1}(\psi)= \frac{\tC+ 2\tA\phi  - \tD\tC\varphi \phi}{\left(1+\tD\varphi\phi\right)^3}.
%\RRR_{121}(\psi)&= D_\phi R_{12}=\frac{C+2 A \phi +AD\phi^2 \varphi}{\left(1+D\varphi\phi\right)^3} \\
%\RRR_{122}(\psi)&= D_\varphi R_{12}=\frac{-CD \phi^2-AD  \phi^3}{\left(1+D\varphi\phi \right)^2}\\
%\RRR_{211}(\psi)&= D_\phi R_{21}=\frac{-CD \varphi^2-AD  \varphi^3}{\left(1+D\varphi\phi \right)^2}\\
%\RRR_{212}(\psi)&= D_\varphi R_{21}=\frac{C+2 \tilde A \varphi +\tilde A D\varphi^2 \phi}{\left(1+D\varphi\phi\right)^3} \\
%\RRR_{221}(\psi)&= D_\phi R_{22}=\frac{C+ 2A\phi +2(B-AD) \varphi \phi^2 - DC\varphi \phi}{\left(1+D\varphi\phi\right)^3}\\
%\RRR_{222}(\psi)&= D_\varphi R_{22}=\frac{2\tilde A\phi +2(\tilde  B-\tilde A D) \phi^2\varphi - 2DC\phi^2}{\left(1+D\varphi\phi\right)^3}
 \]
Moreover, by \eqref{def:psiestimate} and \eqref{def:psisgtarestimate}, we know that 
 $$
 |\psi_*(s) +r (\mathcal{E}^-(s) + t (\mathcal{E}^+(s) - \mathcal{E}^-(s)))|\leq \frac{M_0}{|s|^3}.
 $$
Then, taking into account the definitions of $\tA,\tC,\tD$,
one can obtain the following bounds for $|s|\ge \rho$,
\[
\begin{split}
|\partial_{\phi\phi} \SSS_{1}(\psi)|&\le \frac{1 }{s}
\frac{
2 M_0(1+\frac{1}{\rho})+
2\frac{  M_0^2}{\rho^4}} {\left(1-\frac{M_0^2}{\rr^4}\right)^3}
%=\frac{\overline{\RRR}_{111}}{|s|}
\\
|\partial_{\phi\varphi} \SSS_{1}(\psi)| &\le  \frac{1}{s}
\frac{
1+ 2 M_0(1+\frac{1}{\rho})+
\frac{ M_0^2}{\rho^4}} {\left(1-\frac{M_0^2}{\rr^4}\right)^3}.
%=\frac{\overline{\RRR}_{112}}{|s|}
%\\
%|\RRR_{121}|, |\RRR_{212}| &\le  \frac{b}{s}
%\frac{
%1+ 2(\al+\frac{1}{\rho})M_0+ \frac{ b (\al+\frac{1}{\rho})M_0^3}{\rho^4}} {\left(1-\frac{bM_0^2}{\rr^4}\right)^3}\\
%|\RRR_{122}|, |\RRR_{211}| &\le  \frac{b^2 M_0^2}{s^5}
%\frac{
%1+ (\al+\frac{2}{\rho})M_0} {\left(1-\frac{bM_0^2}{\rr^4}\right)^3}\\
\end{split}
 \]
Using \eqref{eq:Kijaprox} and the bounds \eqref{def:Ephiestimate}, we obtain that $\KK_{11}$ satisfies
\begin{equation}\label{eq:k11error}
\KK_{11}(s)= \partial_{\phi} \SSS_{1}(\psi_*(s))+ \mathcal{EK}_{11}
\end{equation}
with
\begin{equation}\label{errorEK11}
|\mathcal{EK}_{11}|\le \frac{B}{|s|^7}\RRR\qquad \text{where}\qquad \RRR=\frac{
1+ 4 M_0(1+\frac{1}{\rho})+
3\frac{ M_0^2}{\rho^4}} {\left(1-\frac{M_0^2}{\rr^4}\right)^3}.
\end{equation}
Last step is to compute $\partial_{\phi} \SSS_{1}(\psi_*)$ using the formula of $\psi_*$ in Lemma \ref{lemma:psiexpansion}. We recall that
\[
\begin{split}
\partial_{\phi} \SSS_{1}(\psi_*)&=\frac{\varphi_*\phi _* s^2\left(i-\frac{1}{s}\right)\left(2+\varphi_*\phi_* s^2\right)+\frac{1}{s}\varphi_*}{\left(1+\varphi_*\phi_* s^2\right)^2}
\end{split}
\]
and we write
\[
\begin{split}
\partial_{\phi} \SSS_{1}(\psi_*)&=
\frac{2\varphi_*\phi _* s^2\left(i-\frac{1}{s}\right)+\frac{1}{s}\varphi_*}{\left(1+\varphi_*\phi_* s^2\right)^2}+
\mathcal{ER}_1
\end{split}
\]
with
\begin{equation}\label{errorER1}
| \mathcal{ER}_1 |=
\left |\frac{\varphi_*\phi_* s^2\left(i-\frac{1}{s}\right)\varphi_*\phi_* s^2}{\left(1+\varphi_*\phi_* s^2\right)^2} \right |
\le  \frac{M_*^4 (1+\frac{1}{\rho})}{\left(1-\frac{M_*^2}{\rr^4}\right)^2}\frac{1}{|s|^8}
=\frac{B_{11}}{|s|^8}.
\end{equation}
Now, using that
\[
\frac{1}{(1+x)^2}=1-2x+\frac{x^2(3+2x)}{1+x^2},
\]
we write
\[
\begin{split}
\partial_{\phi} \SSS_{1}(\psi_*)&=
\left(2\varphi_*\phi_* s^2\left(i-\frac{1}{s}\right)+\frac{1}{s}\varphi_*\right)
\left(1-2\varphi_*\phi_* s^2\right)+
\mathcal{ER}_1 +\mathcal{ER}_2\\
&=
2\varphi_*\phi_* s^2\left(i-\frac{1}{s}\right)+\frac{1}{s}\varphi_*
+
\mathcal{ER}_1 +\mathcal{ER}_2+\mathcal{ER}_3
\end{split}
\]
with
\begin{equation}\label{errorER23}
\begin{split}
| \mathcal{ER}_2 |&=
\left |2\varphi_*\phi _* s^2\left(i-\frac{1}{s}\right)+\frac{1}{s}\varphi_*\right|
\left|\frac{\varphi_*^2\phi _*^2 s^4\left(3+ 2\varphi_*\phi_* s^2\right)}{\left(1+\varphi_*\phi_* s^2\right)^2}\right|\\
& \le
\frac{M_*^5}{|s|^{12}}\frac{ (2M_*(1+\frac{1}{\rho})+1)(3+\frac{2M_*^2}{\rho^4})}
{\left(1-\frac{M_*^2}{\rr^4}\right)^2}
\\
& \le
\frac{M_*^5}{|s|^{8}}\frac{ (2M_*(1+\frac{1}{\rho})+1)(3+\frac{2M_*^2}{\rho^4})}
{\rho^4\left(1-\frac{M_*^2}{\rr^4}\right)^2}
=\frac{B_{12}}{|s|^8}
\\
|\mathcal{ER}_3|&=\left| \left(2\varphi_*\phi_* s^2\left(i-\frac{1}{s}\right)+\frac{1}{s}\varphi_*\right) \left(-2\varphi_*\phi_* s^2\right)\right|\\
&\le
\frac{2  M_*^3}{|s|^{8}} \left (2M_* \left  (1+\frac{1}{\rho}\right )+1 \right )
=\frac{B_{13}}{|s|^8}.
\end{split}
\end{equation}
We  now substitute the expressions $\psi_*$ in Lemma \ref{lemma:psiexpansion} which give
\[
 \begin{split}
 \phi_*\varphi_*=&\frac{1}{ s^6}-\frac{24}{s^8}+\frac{400}{s^{10}}\\ 
 2\varphi_*\phi_* s^2\left(i-\frac{1}{s}\right)+\frac{1}{s}\varphi_*=& \frac{3i}{ s^4}-\frac{6}{s^5}-\frac{68i}{s^6}+\frac{48}{ s^7}+\frac{800i}{s^8}-\frac{800}{ s^9}
\end{split}
% }
\]
%\[
%\begin{split}
%\varphi_0 &= -\frac{4}{\al^2 s^4}+i \frac{\al^2-20}{\al ^3 s^3}\\
%\varphi_0\phi_0 = &|\varphi_0|^2 = \frac{16}{\al^4 s^8}+ \frac{(\al^2-20)^2}{\al ^6 s^6}\\
%2b\varphi_0\phi _0 s^2\left(i\al-\frac{1}{s}\right)+\frac{1}{s}b\varphi_0
%=&
%\frac{i}{s^4}\frac{b}{\al^5}(\al^2-20)\left(3 \al^2-20\right)-\frac{4b}{\al ^2 s^5}+
%\frac{i\, 32 b  }{\al ^3 s^6}\\
%&-
%\frac{2b (\al^2-20)^2}{\al^6 s^7}-
%\frac{32 b}{\al^4 s^9}
%\end{split}
%\]
which gives
\[
%\begin{split}
\pa_\phi\SSS_{1}(\psi_*)=\frac{3i}{s^4}-\frac{6}{ s^5}-\frac{68i}{ s^6}+\frac{48}{ s^7}\\
+\mathcal{ER}_1 +\mathcal{ER}_2+\mathcal{ER}_3+\mathcal{ER}_4
%\end{split}
\]
and
\begin{equation}\label{errorER4}
|\mathcal{ER}_4|\leq\frac{800}{|s|^8} \left(1-\frac{1}{\rho}\right)  = \frac{B_{14}}{|s|^8}.
\end{equation}
%\[
%\begin{split}
%\RRR_{11}(\psi_0)=&
%\frac{i}{s^4}\frac{b}{\al^5}(\al^2-20)\left(3 \al^2-20\right)-\frac{4b}{\al ^2 s^5}+
%\frac{i\, 32 b  }{\al ^3 s^6}-
%\frac{2b (\al^2-20)^2}{\al^6 s^7}\\
%&+
%\mathcal{ER}_1 +\mathcal{ER}_2+\mathcal{ER}_3+\mathcal{ER}_4
%%\end{split}
%\%]
%and
%\begin{equation}\label{errorER4}
%|\mathcal{ER}_4|=\frac{32 |b|}{\al^4 |s|^9} \le \frac{32 |b|}{\al^4 \rho}
%\frac{1}{ |s|^8}  = \frac{B_{14}}{|s|^8}.
%\end{equation}
Using these approximations in \eqref{eq:k11error}, we obtain the statement of the lemma
%conclude that
%\[\KK_{11}(s)=
%\frac{i}{s^4}\frac{b}{\al^5}(\al^2-20)\left(3 \al^2-20\right)-\frac{4b}{\al ^2 s^5}+
%\frac{i\, 32 b  }{\al ^3 s^6}-\frac{2b (\al^2-20)^2}{\al^6 s^7}+ \mathcal{EKT}_{11}\]
%\begin{equation}\label{formulaK11}
%\begin{split}
%\KK_{11}(s) &=
%%\frac{i\beta_1 }{s^4} +\frac{\beta_2}{ s^5}+
%\frac{i\, \beta_3  }{ s^6}+\frac{\beta_4}{ s^7}
%+ \mathcal{EKT}_{11}\\
%\beta_1&=\frac{b}{\al^5}(\al^2-20)\left(3 \al^2-20\right), \ \beta_2= -\frac{4b}{\al ^2 }, \\
%\beta _3&=\frac{32 b  }{\al ^3}, \beta_4=-
%\frac{2b (\al^2-20)^2}{\al^6}
%\end{split}
%\end{equation}
taking
\[
\mathcal{EKT}_{11}= \mathcal{EK}_{11}
+
\mathcal{ER}_1 +\mathcal{ER}_2+\mathcal{ER}_3+\mathcal{ER}_4.
\]
Using the bounds~\eqref{errorEK11}, \eqref{errorER1}, \eqref{errorER23}, \eqref{errorER4} we get
\begin{equation}\label{errorK11total:1}
|\mathcal{EKT}_{11}| \le \frac{ B\RRR}{|s|^7}+\frac{B_{11}+B_{12}+B_{13}+B_{14}}{|s|^8}
\end{equation}.

\section*{Acknowledgements}
This project has received funding from the European Research Council (ERC) under the European Union’s
Horizon 2020 research and innovation programme (grant agreement No 757802). 
This work is part of the grant PGC2018-098676-B-100 funded by MCIN/AEI/10.13039/501100011033 and “ERDF A way of making Europe”.
%I.B and T. M. S. has also been partly supported by the Spanish MINECO-FEDER Grant PGC2018-098676-B-100 (AEI/FEDER/UE) %
I.B and T. M. S. has also been partly supported by the and the Catalan grant 2017SGR1049. M. G.
and T. M. S. are supported by the Catalan Institution for Research and Advanced Studies via an ICREA
Academia Prize 2019.  M.C. has been partially supported by the Polish National Science Center (NCN) grants 2019/35/B/ST1/00655 and 2021/41/B/ST1/00407.
This work is also supported by the Spanish State Research Agency, through the Severo Ochoa and María de Maeztu Program for Centers and Units of Excellence in R\&D (CEX2020-001084-M).

%\section*{References}
\appendix
\section{Proof of Lemma \ref{cor:linearB}}\label{app:ProofOperator}
To prove Lemma \ref{cor:linearB}, we need  first the following lemma.
\begin{lemma} \label{lem:Im}
If $m$ is even
$$
\int_{-\infty}^0 \frac{1}{(t^2 + 1)^{m/2}} \, dt = \frac{\pi}{2} \frac{(m-3))!!}{(m-2)!!} =:B_{m}
$$
and, for $m$ odd
$$
\int_{-\infty}^0 \frac{1}{(t^2 + 1)^{m/2}} \, dt = \frac{(m-3)!!}{(m-2)!!} =:B_m
$$
\end{lemma}
\begin{proof}
Integrating by parts,
\[
\begin{split}
I_{k}:&= \int_{-\infty}^0 \frac{1}{(t^2 + 1)^{k/2}} \, dt = k\int_{-\infty}^0 \frac{t^2}{(t^2 + 1)^{\frac{k}{2}+1}} \\
&= kI_k - kI_{k+2}.
\end{split}
\]
Therefore $I_{k+2}= \frac{k-1}{k} I_{k}$ which implies $I_{k}=\frac{k-3}{k-2}I_{k-2}$ and therefore
$$
I_{k}=\frac{(k-3)!!}{(k-2)!!} J,
$$
with $J=I_3$ if $k$ is odd and $J=I_2$ if $k$ is even.
Since
$
I_3=1$ and $ I_2=\frac{\pi}{2}
$
we are done.
\end{proof}

We use this lemma to prove Lemma \ref{cor:linearB}.
\begin{proof}[Proof of Lemma \ref{cor:linearB}]
Let $\psi =(\phi,\varphi)\in \XX_{\nu}$
The first component of $\B^-(\psi)$ is
\begin{equation}\label{B1}
\B^-_1(\psi) = s \int_{-\infty}^0 \frac{e^{i\al t}}{s+t}\phi(s+t)\, dt.
\end{equation}
We first prove the first item. Indeed, using that 
$|s+t|^2 \geq |s|^2 + t^2$ for  $s\in \DDD^-_\rr$, one can prove that 
$$
\left | \B^-_1(\psi)\right|\leq |s| \|\phi\|_\nu \int_{-\infty}^0 \frac{1}{|s+t|^{\nu+1}}\, dt \leq \frac{\|\phi\|_{\nu}}{|s|^{\nu-1}}
\int_{-\infty}^0 \frac{1}{\big (t^2 + 1\big )^{\frac{\nu+1}{2}}} = B_{\nu+1} \frac{\|\phi\|_{\nu}}{|s|^{\nu-1}}.
$$
Analogously we deal with $\B^-_2(\psi)$ and we obtain the result in the first item, taking into account that the product norm is the supremum norm.

Now we deal with the second item. By the geometry of $\DDD^-_\rr$, and using the Cauchy's theorem, we can change the path of integration in the integral ~\eqref{B1} defining $\B_1^-(\psi)$ to $t e^{i\gamma}$, $t\in(-\infty,0]$, with $0\leq \gamma \leq \beta$. We obtain then
$$
\B^-_1(\psi)(s) = s \int_{-\infty}^0 \frac{e^{i\al te^{i\gamma}}}{s+te^{i\gamma}} \phi(s+te^{i\gamma})e^{i\gamma} \, dt.
$$
Notice that $s+te^{i\gamma} \in \DDD^-_\rr$ and
$$
|s+te^i\gamma| \geq |s|\sin\left (\frac{\pi}{2}-\gamma\right ) = |s|\cos \gamma.
$$
Therefore
$$
\left | \B^-_1(\psi)\right | \leq \frac{\|\phi\|_\nu }{|s|^\nu \big (\cos \gamma \big )^{\nu+1}}\int_{-\infty}^0
e^{\al \sin \gamma t} \, dt =  \frac{\|\phi\|_\nu}{|s|^\nu \al \big (\cos \gamma\big )^{\nu+1} \sin \gamma}.
$$
The function $\big (\cos \gamma \big )^{\nu+1} \sin \gamma$ has only a maximum in $\left (0,\frac{\pi}{2}\right )$ in $\gamma_*$ such that $(\nu+1)\sin^2 \gamma_* =1$.

As in the first item, $\B_2^-(\psi)$ can be treated in the same way, changing here the integration path to
$te^{-i\gamma}$.
\end{proof}

\section{Computing the bound on $\Delta\psi\left(  -i\rho^{\ast
}\right)  $\label{app:Delta-comp}}

Here we provide an explicit rigorous estimate, using interval
arithmetic bounds, for the distance between $\psi^{+}\left(  -i\rho\right)  $
and $\psi^{-}\left(  -i\rho\right)  $ for a given $\rho>0$ which we have used for our Examples 1 and 2 (discussed in Sections  \ref{sec:example1}, \ref{sec:example2} and \ref{sec:StokesComputeExtrap}).

We start with Example 1. We work with the system (\ref{eq:s-system})
with $F_{1}=-s^{-3}$, $F_{2}=0,$ $H=0$.
%, prior to the change to complex time.

By writing
\[
\phi=x_{1}+iy_{1},\qquad\varphi=x_{2}+iy_{2},\qquad s=s_{1}+is_{2},
\]
we can rewrite (\ref{eq:s-system}) as%
\begin{align}
x_{1}^{\prime} &  =\left(  1+\frac{s_{2}}{s_{1}^{2}+s_{2}^{2}}\right)
y_{1}+\frac{s_{1}x_{1}}{s_{1}^{2}+s_{2}^{2}}-\frac{s_{1}^{3}-3s_{1}s_{2}^{2}%
}{\left(  s_{1}^{2}+s_{2}^{2}\right)  ^{3}},\nonumber\\
y_{1}^{\prime} &  =-\left(  1+\frac{s_{2}}{s_{1}^{2}+s_{2}^{2}}\right)
x_{1}+\frac{s_{1}y_{1}}{s_{1}^{2}+s_{2}^{2}}-\frac{s_{2}^{3}-3s_{1}^{2}s_{2}%
}{\left(  s_{1}^{2}+s_{2}^{2}\right)  ^{3}},\nonumber\\
x_{2}^{\prime} &  =-\left(  1-\frac{s_{2}}{s_{1}^{2}+s_{2}^{2}}\right)
y_{2}+\frac{s_{1}x_{2}}{s_{1}^{2}+s_{2}^{2}}-\frac{s_{1}^{3}-3s_{1}s_{2}^{2}%
}{\left(  s_{1}^{2}+s_{2}^{2}\right)  ^{3}},\label{eq:ODE-ex1-real-time}\\
y_{2}^{\prime} &  =\left(  1-\frac{s_{2}}{s_{1}^{2}+s_{2}^{2}}\right)
x_{2}+\frac{s_{1}y_{2}}{s_{1}^{2}+s_{2}^{2}}-\frac{s_{2}^{3}-3s_{1}^{2}s_{2}%
}{\left(  s_{1}^{2}+s_{2}^{2}\right)  ^{3}},\nonumber\\
s_{1}^{\prime} &  =-2s_{1}s_{2}\left(  x_{1}y_{2}+x_{2}y_{1}\right)  +\left(
x_{1}x_{2}-y_{1}y_{2}\right)  \left(  s_{1}^{2}-s_{2}^{2}\right)
+1,\nonumber\\
s_{2}^{\prime} &  =2s_{1}s_{2}\left(  x_{1}x_{2}-y_{1}y_{2}\right)  +\left(
x_{1}y_{2}+x_{2}y_{1}\right)  \left(  s_{1}^{2}-s_{2}^{2}\right)  .\nonumber
\end{align}
This is an ODE in $\mathbb{R}^{6}$, with real time $\tau$. We shall write
$\Phi_{\tau}$ for the flow induced by (\ref{eq:ODE-ex1-real-time}).

Note that $\left(  s_{1},s_{2}\right)  =\left(  0,-\rho\right)$ corresponds
to the complex $s=-i\rho$.

For $x\in\mathbb{R}^{6}$ we shall write%
\begin{align*}
\tau^{+}\left(  x\right)   &  =\sup\left\{  \tau<0:\pi_{s_{1}}\Phi_{\tau}%
(x)=0\right\}  ,\\
\tau^{-}\left(  x\right)   &  =\inf\left\{  \tau>0:\pi_{s_{1}}\Phi_{\tau}%
(x)=0\right\}  ,
\end{align*}
and define
\begin{align*}
\mathcal{P}^{+},\mathcal{P}^{-} &  :\mathbb{R}^{6}\rightarrow\left\{
s_{1}=0\right\}  ,\\
\mathcal{P}^{+}\left(  x\right)   &  :=\Phi_{\tau^{+}\left(  x\right)  }(x),\\
\mathcal{P}^{-}\left(  x\right)   &  :=\Phi_{\tau^{-}\left(  x\right)  }(x).
\end{align*}
We do not assume that $\mathcal{P}^{+}$ and $\mathcal{P}^{-}$ are globally
defined. Whenever we write $\mathcal{P}^{+}\left(  x\right)  $ or
$\mathcal{P}^{-}\left(  x\right)  $ we will always validate that the
considered point $x$ lies in the domain of the map.

We know that the two solutions $\psi^{\pm}$ of \eqref{eq:Inner} (with  $F_{1}=-s^{-3}$, $F_{2}=0,$ $H=0$) satisfy%
\begin{align}
\left\Vert \psi^{-}(s)\right\Vert  & \leq\left\vert s\right\vert ^{-3}%
M_{0}\leq\left\vert
%TCIMACRO{\TeXButton{Re}{\Re}}%
%BeginExpansion
\Re
%EndExpansion
s\right\vert ^{-3}M_{0}\qquad\text{for }%
%TCIMACRO{\TeXButton{Re}{\Re}}%
%BeginExpansion
\Re
%EndExpansion
s<0,\label{eq:psi-minus-M0-bound}\\
\left\Vert \psi^{+}(s)\right\Vert  & \leq\left\vert s\right\vert ^{-3}%
M_{0}\leq\left\vert
%TCIMACRO{\TeXButton{Re}{\Re}}%
%BeginExpansion
\Re
%EndExpansion
s\right\vert ^{-3}M_{0}\qquad\text{for }%
%TCIMACRO{\TeXButton{Re}{\Re}}%
%BeginExpansion
\Re
%EndExpansion
s>0.\label{eq:psi-plus-M0-bound}%
\end{align}

\begin{figure}
    \centering
    \includegraphics[height=3.5cm]{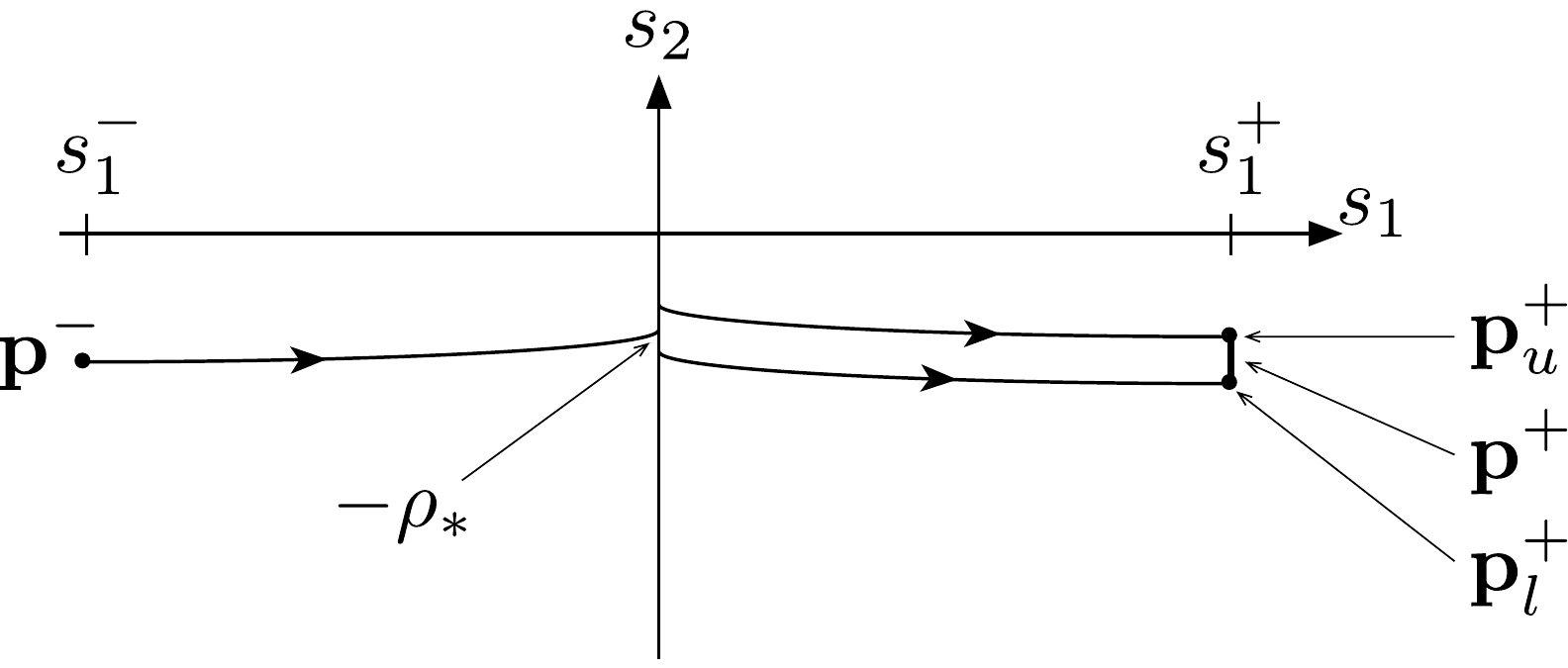}
    \caption{Illustration for Lemma \ref{lem:interval-tool}. Here we depict the projection of the flow onto the $s=(s_1,s_2)$ coordinates. The sets $\mathcal{P}^+(\mathbf{p}^+)$ and $\mathcal{P}^-(\mathbf{p}^-)$ which lead to the bound (\ref{eq:delta-rho-bound}) for $\Delta \psi$ lie on the section $\{s_1=0\}$. }
    \label{fig:Delta-psi}
\end{figure}

\begin{lemma}
\label{lem:interval-tool}Consider $s^{-},s_{l}^{+},s_{u}^{+}\in\mathbb{R}^{2}$
of the form $s^{-}=(  s_{1}^{-},s_{2}^{-})  ,$ $s_{l}^{+}=(
s_{1}^{+},s_{2,l}^{+})  ,$ $s_{u}^{+}=(  s_{1}^{+},s_{2,u}%
^{+})  $, where $s_{1}^{-}<0<s_{1}^{+}$ and $s_{2,l}^{+}<s_{2,u}^{+}$.
(The subscripts $l$ and $u$ stand for `lower' and `upper'; see Figure \ref{fig:Delta-psi}.)
Let $\mathbf{s}^{+}\subset\mathbb{R}^{2}$ be the vertical interval joining $s_{l}^{+}$
and $s_{u}^{+}$ and let%
\begin{align*}
\mathbf{p}^{-}  & :=\left\{  p\in\mathbb{R}^{6}:\pi_{s}p=s^{-},\left\Vert
\pi_{x,y}p\right\Vert \leq\left\vert s_{1}^{-}\right\vert ^{-3}M_{0}\right\}
,\\
\mathbf{p}^{+}  & :=\left\{  p\in\mathbb{R}^{6}:\pi_{s}p\in\mathbf{s}%
^{+},\left\Vert \pi_{x,y}p\right\Vert \leq\left\vert s_{1}^{+}\right\vert
^{-3}M_{0}\right\}  ,\\
\mathbf{p}_{u}^{+}  & :=\left\{  p\in\mathbb{R}^{6}:\pi_{s}p=s_{u}%
^{+},\left\Vert \pi_{x,y}p\right\Vert \leq\left\vert s_{1}^{+}\right\vert
^{-3}M_{0}\right\}  ,\\
\mathbf{p}_{l}^{+}  & :=\left\{  p\in\mathbb{R}^{6}:\pi_{s}p=s_{l}%
^{+},\left\Vert \pi_{x,y}p\right\Vert \leq\left\vert s_{1}^{+}\right\vert
^{-3}M_{0}\right\}  .
\end{align*}
If
\begin{equation}
\pi_{s_{2}}\mathcal{P}^{+}\left(  \mathbf{p}_{l}^{+}\right)  <\pi_{s_{2}%
}\mathcal{P}^{-}\left(  \mathbf{p}^{-}\right)  <\pi_{s_{2}}\mathcal{P}%
^{+}\left(  \mathbf{p}_{u}^{+}\right)  <0\label{eq:Bolzano-ineq}%
\end{equation}
then there exists a $\rho^{\ast}\in-\pi_{s_{2}}\mathcal{P}^{-}\left(
\mathbf{p}^{-}\right)  $ such that%
\begin{equation}
\Delta\psi\left(  -i\rho^{\ast}\right)  =\psi^{+}\left(  -i\rho^{\ast}\right)
-\psi^{-}\left(  -i\rho^{\ast}\right)  \in\pi_{x,y}\left(  \mathcal{P}%
^{+}\left(  \mathbf{p}^{+}\right)  -\mathcal{P}^{-}\left(  \mathbf{p}%
^{-}\right)  \right)  .\label{eq:delta-rho-bound}%
\end{equation}

\end{lemma}

\begin{proof}
From (\ref{eq:psi-minus-M0-bound}--\ref{eq:psi-plus-M0-bound}) we see that
\begin{align*}
\psi^{-}\left(  s^{-}\right)    & \in\pi_{x,y}\mathbf{p}^{-},\\
\psi^{+}\left(  \mathbf{s}^{+}\right)    & \subset\pi_{x,y}\mathbf{p}^{+}.
\end{align*}
By the Bolzano theorem applied to
\[
\mathbf{s}^{+}\ni s^{+}\mapsto\pi_{s_{2}}\left(  \mathcal{P}^{+}\left(
\psi^{-}\left(  s^{+}\right)  ,s^{+}\right)  -\mathcal{P}^{-}\left(  \psi
^{-}\left(  s^{-}\right)  ,s^{-}\right)  \right)
\]
from (\ref{eq:Bolzano-ineq}) we see that there there exists a $s_{\ast}^{+}%
\in\mathbf{s}^{+}$ such that%
\[
\mathcal{P}^{+}\left(  \psi^{-}\left(  s_{\ast}^{+}\right)  ,s_{\ast}%
^{+}\right)  =\mathcal{P}^{-}\left(  \psi^{-}\left(  s^{-}\right)
,s^{-}\right)  .
\]
By definition of $\mathcal{P}^{\pm}$ we know that $\pi_{s}\mathcal{P}^{\pm}\left(  q\right)  \in\left\{  0\right\}
\times\mathbb{R}$, so%
\[
\pi_{s}\mathcal{P}^{+}\left(  \psi^{-}\left(  s_{\ast}^{+}\right)  ,s_{\ast
}^{+}\right)  =\pi_{s}\mathcal{P}^{-}\left(  \psi^{-}\left(  s^{-}\right)
,s^{-}\right)  =\left(  0,-\rho^{\ast}\right)  ,
\]
for some $\rho^{\ast}>0$ (the sign follows from (\ref{eq:Bolzano-ineq})) and
hence%
\begin{align*}
\psi^{-}\left(  -i\rho^{\ast}\right)    & =\pi_{x,y}\mathcal{P}^{-}\left(
\psi^{-}\left(  s^{-}\right)  ,s^{-}\right)  \in\pi_{x,y}\mathcal{P}%
^{-}\left(  \mathbf{p}^{-}\right)  ,\\
\psi^{+}\left(  -i\rho^{\ast}\right)    & =\pi_{x,y}\mathcal{P}^{-}\left(
\psi^{-}\left(  s_{\ast}^{-}\right)  ,s_{\ast}^{-}\right)  \in\pi
_{x,y}\mathcal{P}^{+}\left(  \mathbf{p}^{+}\right)  ,
\end{align*}
which implies (\ref{eq:delta-rho-bound}), as required.
\end{proof}

In our computer assisted proof we have taken $\bar{\rho}:=$%
\texttt{16.00008679} and
\begin{align}
s^{-}  & =\left(  -10^{3},-\bar{\rho}\right)  , \label{eq:s-minus}\\
s_{u}^{+}  & =\left(  10^{3},-\bar{\rho}+10^{-6}\right)  ,\\
s_{l}^{+}  & =\left(  10^{3},-\bar{\rho}-10^{-6}\right) \label{eq:s-plus-l} .
\end{align}
(the choice of $\bar{\rho}$ is dictated by the fact that then $\rho^{\ast
}\approx16$; see (\ref{eq:rho-ex1})). Then, we have  validated that, with such choice of
$s^{-},s_{u}^{+},s_{l}^{+}$, Lemma \ref{lem:interval-tool} leads to the bound
(\ref{eq:Delta-ex1}). The computation of $\mathcal{P}^{+}\left(
\mathbf{p}^{+}\right)  ,\mathcal{P}^{-}\left(  \mathbf{p}^{-}\right)  $
required a long integration time, due to the number $10^{3}$ in our choice of
$s^{-}_1,s_1^{+}$. The benefit of such large value in $%
%TCIMACRO{\TeXButton{Re}{\Re}}%
%BeginExpansion
\Re
%EndExpansion
s$ is that then $\left\vert
%TCIMACRO{\TeXButton{Re}{\Re}}%
%BeginExpansion
\Re
%EndExpansion
s\right\vert ^{-3}$ is a very small number, leading to small sets
$\mathbf{p}^{\pm},\mathbf{p}_{l}^{+},\mathbf{p}_{u}^{+}$. This results in good
bounds on $\Delta\psi$. Such choice of $%
%TCIMACRO{\TeXButton{Re}{\Re}}%
%BeginExpansion
\Re
%EndExpansion
s$ was reached by trial and error.

The computer assisted validation of (\ref{eq:Delta-ex1}) took under 20
seconds, running on a single thread of a standard laptop.

The computation of $\Delta\psi$ in the second example also follows from Lemma
\ref{lem:interval-tool}. The only difference is the formula for the vector
field, which is%
\begin{align*}
x_{1}^{\prime} &  =\frac{s_{1}}{s_{1}^{2}+s_{2}^{2}}x_{1}+\left(  \alpha
+\frac{s_{2}}{s_{1}^{2}+s_{2}^{2}}\right)  y_{1}-\frac{s_{1}^{3}-3s_{1}%
s_{2}^{2}}{\left(  s_{1}^{2}+s_{2}^{2}\right)  ^{3}}\\
&  +\frac{1}{s_{1}^{2}+s_{2}^{2}}\left(  s_{2}\left(  x_{1}^{2}-y_{1}%
^{2}-x_{2}^{2}+y_{2}^{2}\right)  +2s_{1}\left(  x_{2}y_{2}-x_{1}y_{1}\right)
\right)  ,\\
y_{1}^{\prime} &  =-\left(  \alpha+\frac{s_{2}}{s_{1}^{2}+s_{2}^{2}}\right)
x_{1}+\frac{s_{1}}{s_{1}^{2}+s_{2}^{2}}y_{1}-\frac{s_{2}^{3}-3s_{1}^{2}s_{2}%
}{\left(  s_{1}^{2}+s_{2}^{2}\right)  ^{3}}\\
&  +\frac{1}{s_{1}^{2}+s_{2}^{2}}\left(  s_{1}\left(  x_{1}^{2}-x_{2}%
^{2}-y_{1}^{2}+y_{2}^{2}\right)  +2s_{2}\left(  x_{1}y_{1}-x_{2}y_{2}\right)
\right)  ,\\
x_{2}^{\prime} &  =\frac{s_{1}}{s_{1}^{2}+s_{2}^{2}}x_{2}-\left(  \alpha
-\frac{s_{2}}{s_{1}^{2}+s_{2}^{2}}\right)  y_{2}-\frac{s_{1}^{3}-3s_{1}%
s_{2}^{2}}{\left(  s_{1}^{2}+s_{2}^{2}\right)  ^{3}}\\
&  +\frac{1}{s_{1}^{2}+s_{2}^{2}}\left(  s_{2}\left(  x_{1}^{2}-y_{1}%
^{2}-x_{2}^{2}+y_{2}^{2}\right)  +2s_{1}\left(  x_{2}y_{2}-x_{1}y_{1}\right)
\right)  ,\\
y_{2}^{\prime} &  =\left(  \alpha-\frac{s_{2}}{s_{1}^{2}+s_{2}^{2}}\right)
x_{2}+\frac{s_{1}}{s_{1}^{2}+s_{2}^{2}}y_{2}-\frac{s_{2}^{3}-3s_{1}^{2}s_{2}%
}{\left(  s_{1}^{2}+s_{2}^{2}\right)  ^{3}}\\
&  +\frac{1}{s_{1}^{2}+s_{2}^{2}}\left(  s_{1}\left(  x_{1}^{2}-x_{2}%
^{2}-y_{1}^{2}+y_{2}^{2}\right)  +2s_{2}\left(  x_{1}y_{1}-x_{2}y_{2}\right)
\right)  ,\\
s_{1}^{\prime} &  =1-2bs_{1}s_{2}\left(  x_{1}y_{2}+x_{2}y_{1}\right)
+b\left(  s_{1}^{2}-s_{2}^{2}\right)  \left(  x_{1}x_{2}-y_{1}y_{2}\right)
,\\
s_{2}^{\prime} &  =2bs_{1}s_{2}\left(  x_{1}x_{2}-y_{1}y_{2}\right)  +b\left(
s_{1}^{2}-s_{2}^{2}\right)  \left(  x_{1}\allowbreak y_{2}+x_{2}y_{1}\right)
.
\end{align*}

We take the same $s^{-},s_{l}^{+},s_{u}^{+}$ as in (\ref{eq:s-minus}%
--\ref{eq:s-plus-l}) which, with the aid of Lemma \ref{lem:interval-tool} and
interval arithmetic integration, leads to the bounds (\ref{eq:rho-ex2}%
--\ref{eq:Delta-ex2}).

The computer assisted validation of (\ref{eq:Delta-ex2}) took under 25 seconds, running on
a single thread of a standard laptop.
\bibliography{HopfZeroEstimatexbiblio.bib}
\bibliographystyle{alpha}
\end{document}